\documentclass[10pt]{article}

\usepackage{microtype}
\usepackage{graphicx}
\usepackage{geometry}   %
\usepackage{subfigure}
\usepackage{booktabs} %
\usepackage[pagebackref]{hyperref}       %
\hypersetup{
  colorlinks=true,
  linkcolor=blue,
  filecolor=magenta,
  urlcolor=cyan,
  citecolor=blue
}
 \renewcommand*{\backrefalt}[4]{%
     \ifcase #1 \footnotesize{(Not cited).}%
     \or        \footnotesize{(Cited on page~#2).}%
     \else      \footnotesize{(Cited on pages~#2).}%
     \fi}

\usepackage{authblk}
\RequirePackage{fancyhdr}
\RequirePackage{xcolor} %
\RequirePackage{algorithm}
\RequirePackage{algorithmic}
\RequirePackage{natbib}
\RequirePackage{eso-pic} %
\RequirePackage{forloop}

\setcitestyle{authoryear,round,citesep={;},aysep={,},yysep={;}}

\usepackage{amsmath}
\usepackage{amssymb}
\usepackage{mathtools}
\usepackage{amsthm}

\usepackage{thmtools}
\usepackage{thm-restate}

\declaretheorem{theorem}
\declaretheorem{corollary}
\declaretheorem[sibling=corollary]{proposition}
\declaretheorem{lemma}
\declaretheorem[numberwithin=section]{definition, assumption, remark}
\usepackage[noabbrev]{cleveref}
\crefname{assumption}{assumption}{assumptions}
\usepackage[disable,textsize=tiny]{todonotes}

\usepackage{rka-style}
\renewcommand{\bar}{\overline}

\title{\bf Tuning-Free Stochastic Optimization}
\author[1]{Ahmed Khaled}
\author[1]{Chi Jin}
\affil[1]{Princeton University, Princeton, NJ, USA}

\begin{document}
\maketitle

\begin{abstract}
Large-scale machine learning problems make the cost of hyperparameter tuning
ever more prohibitive. This creates a need for algorithms that can tune
themselves on-the-fly. We formalize the notion of \emph{``tuning-free''}
algorithms that can match the performance of optimally-tuned optimization
algorithms up to polylogarithmic factors given only loose hints on the relevant
problem parameters. We consider in particular algorithms that can match
optimally-tuned Stochastic Gradient Descent (SGD). When the domain of
optimization is bounded, we show tuning-free matching of SGD is possible and
achieved by several existing algorithms. We prove that for the task of
minimizing a convex and smooth or Lipschitz function over an unbounded domain,
tuning-free optimization is impossible. We discuss conditions under which
tuning-free optimization is possible even over unbounded domains. In particular,
we show that the recently proposed DoG and DoWG algorithms are tuning-free when
the noise distribution is sufficiently well-behaved. For the task of finding a
stationary point of a smooth and potentially nonconvex function, we give a
variant of SGD that matches the best-known high-probability convergence rate for
tuned SGD at only an additional polylogarithmic cost. However, we also give an
impossibility result that shows no algorithm can hope to match the optimal
expected convergence rate for tuned SGD with high probability.
\end{abstract}

\section{Introduction}

The hyperparameters we supply to an optimization algorithm can have a
significant effect on the runtime of the algorithm and the quality of the final
model~\citep{yang2021tuning,sivaprasad19_optim_bench_needs_to_accoun_hyper_tunin}.
Yet hyperparameter tuning is costly, and for large models might prove
intractable~\citep{black22_gpt_neox}. As a result, researchers often resort to
using a well-known optimizer like Adam~\citep{kingma14_adam} or
AdamW~\citep{loshchilov17_decoup_weigh_decay_regul} with widely used or default
hyperparameters. For example,
GPT-3~\citep{brown20_languag_model_are_few_shot_learn},
BLOOM~\citep{workshop22_bloom}, LLaMA~\citep{touvron23_llamaone}, and
LLaMA2 \citep{touvron23_llamatwo} all use either Adam or AdamW with identical
momentum parameters and similar training recipes.

This situation presents an immense opportunity for algorithms that
can tune hyperparameters on-the-fly. Yet such algorithms and their limits
are still poorly understood in the setting of stochastic optimization. Let us
make our setting more specific. We consider the minimization problem
\begin{align}
\label{eq:opt-problem}\tag{OPT}
\min_{x \in \mathcal{X}} f(x),
\end{align}
where $f: \mathcal{X} \rightarrow \mathbb{R}$ is differentiable and lower bounded by $f_{\ast}$. We assume that
we have access to (stochastic) gradients $g(x)$ that satisfy certain regularity
conditions that we shall make precise later.

Our main objects of study are \emph{tuning-free} algorithms. To make this notion
more precise, let $\mathcal{A}$ be an optimization algorithm that takes in $n$ problem
parameters $a = (a_1, a_2, \ldots, a_n)$ and after $T$ (stochastic) gradient accesses
returns a point $\overline{x}$ such that with high probability
\begin{align}
\label{eq:49}
f(\overline{x}) - f_{\ast} \leq \mathrm{Error}_{\mathcal{A}} (f, a, T).
\end{align}
The function $\mathrm{Error}_{\mathcal{A}}$ characterizes how well the algorithm $\mathcal{A}$
minimizes the function $f$ in $T$ steps given the supplied parameters. Let
$a^{\ast} = a^{\ast} (f, T)$ denote the set of parameters that minimizes the right hand side
of \cref{eq:49} for a specific function $f$ and number of steps $T$. In order
for an algorithm to find $a^{\ast} (f, T)$, it must start \emph{somewhere}. We
assume that we can easily find lower and upper bounds on the optimal parameters:
two sets $\underline{a}$ and $\overline{a}$ such that for $i = 1, 2, \ldots, n$ we
have
\begin{align*}
\underline{a}_i \leq a_i^{\ast} \leq \overline{a}_i.
\end{align*}
Such \emph{hints} on problem parameters can often be easily estimated in
practice, and are a much easier ask than the optimal parameters. To be a
\emph{tuning-free} version of $\mathcal{A}$, an algorithm $\mathcal{B}$ has to approximately match
the performance of $\mathcal{A}$ with optimally tuned parameters given those hints, a
definition we make rigorous next.

\begin{definition}\label{def:tuning-free-general}
(\textbf{Tuning-free algorithms}). We call an
algorithm $\mathcal{B}$ a \emph{tuning-free} version of $\mathcal{A}$ if given hints
$\underline{a}, \overline{a}$ on the optimal parameters $a^{\ast}$ for a function
$f$ it achieves the same error as $\mathcal{A}$ with the optimal parameters up to only
polylogarithmic degradation that depends on the hints and the number of
(stochastic) gradient accesses $T$. That is, if $\mathcal{A}$ achieves error
$f(\overline{x}) - f_{\ast} \leq \mathrm{Error}_{\mathcal{A}} (f, a^{\ast} (f, T), T)$, then $\mathcal{B}$
achieves the guarantee:
\begin{align}\tag{$\mathcal{B}$-error}
\label{eq:B-error} f(\overline{x}) - f_{\ast} &\leq \iota \cdot \mathrm{Error}_{\mathcal{A}} (f, a^{\ast}, T),
\end{align} where
$\iota = \mathrm{poly\log}\left(\frac{\overline{a}_1}{\underline{a}_1}, \ldots, \frac{\overline{a}_n}{\underline{a}_n}, T\right)$
is a polylogarithmic function of the hints.
\end{definition}

Clearly, asking a tuning-free algorithm $\mathcal{B}$ to achieve exactly the same error as
$\mathcal{A}$ is too much: we ought to pay some price for not knowing $a^{\ast}$ upfront. On
the other hand, if we allow polynomial dependencies on the hints, then our hints
have to be very precise to avoid large errors. This beats the point of being
tuning-free in the first place.

The algorithm $\mathcal{A}$ that we are primarily concerned with is Stochastic Gradient
Descent (SGD). SGD and its variants dominate in practice, owing to their
scalability and low memory
requirements~\citep{bottou16_optim_method_large_scale_machin_learn}. We consider
three classes of functions: (a) $L$-smooth and convex functions, (b)
$G$-Lipschitz and convex functions, and (c) $L$-smooth and potentially nonconvex
functions. We ask for tuning-free algorithms for each of these function classes.
We give precise definitions of these classes and our oracle model later in
\Cref{sec:preliminaries}.

In the setting of deterministic optimization, we have a very good understanding
of tuning-free optimization: there are many methods that, given only hints on
the problem parameters required by SGD, achieve the same rate as SGD up to only
polylogarithmic degradation. We review this case briefly in
\Cref{sec:tuning-free-optim-unbounded}. Despite
immense algorithmic developments in stochastic optimization~\citep{ duchi11_adagrad, levy17_onlin_to_offlin_conver_univer, levy18_onlin_adapt_method_univer_accel, li18_conver_stoch_gradien_descen_with_adapt_steps, kavis19_unixg, carmon22_makin_sgd_param_free, ivgi23_dog_is_sgds_best_frien, cutkosky23_mechan}
and the related setting of online
learning~\citep{orabona16_coin_bettin_param_free_onlin_learn,cutkosky17_onlin_convex_optim_with_uncon_domain_losses,cutkosky19_artif_const_lipsc_hints_uncon_onlin_learn,mhammedi19_lipsc_adapt_with_multip_learn,mhammedi20_lipsc_compar_norm_adapt_onlin_learn,orabona2020parameter} we are not
aware of \emph{any} algorithms that fit our definition as tuning-free counterparts of
SGD for \emph{any} of the function classes we
consider. The main question of our work is thus:

\begin{quote}
Can we find tuning-free counterparts for SGD in the setting of
stochastic optimization and the classes of functions we consider (convex and
smooth functions, convex and Lipschitz functions, and nonconvex and smooth functions)?
\end{quote}

\textbf{Our contributions.} We answer the above question in the negative for the
first two function classes and make some progress towards answering it for the
third. In particular, our main contributions are:

\begin{itemize}
\item For \textbf{convex optimization}: if the domain of
optimization is bounded, we highlight results from the literature showing
tuning-free optimization matching SGD is possible. If the domain of optimization
$\mathcal{X}$ is unbounded, we give an \textbf{impossibility result} that shows no
algorithm can be a tuning-free counterpart of SGD for smooth and convex
functions (\Cref{thm:impossibility-result-smooth}), as well as for Lipschitz and
convex functions (\Cref{thm:impossibility-result-non-smooth}). Additionally if
the stochastic gradient noise has a certain large \emph{signal-to-noise} ratio
(defined in \Cref{sec:what-noise-distr}), then tuning-free optimization is
possible even when the domain of optimization $\mathcal{X}$ is unbounded and can be
achieved by the recently-proposed DoG~\citep{ivgi23_dog_is_sgds_best_frien} and
DoWG~\citep{khaled23_dowg_unleas} algorithms for both smooth and/or Lipschitz
functions (\Cref{thm:dog-dowg-special-noise}).
\item For \textbf{nonconvex optimization}: We consider two different notions of
tuning-free optimization that correspond to the best-known convergence error
bounds for SGD in expectation~\citep{ghadimi2013stochastic} and with high
probability~\citep{liu23_high_probab_conver_stoch_gradien_method}. We show
tuning-free optimization is \textbf{impossible} in the former setting
(\Cref{thm:impossibility-result-nonconvex}). On the other hand, for the latter,
slightly weaker notion, we give a \textbf{positive result} and give a
tuning-free variant of SGD~(\Cref{thm:restarted-sgd}).
\end{itemize}

\subsection{Preliminaries}
\label{sec:preliminaries}
In this section we review some preliminary notions and definitions that we shall
make use of throughout the paper. We say that a function $f: \mathcal{X} \rightarrow \mathbb{R}$ is convex if
for any $x, y \in \mathcal{X}$ we have
\begin{align*} f(tx + (1-t) y) \leq t f(x) + (1 - t) f(y) \text { for all } t \in [0, 1].
\end{align*} We call a function $G$-Lipschitz if
$\abs{f(x) - f(y)} \leq G \norm{x-y}$ for all $x, y \in \mathcal{X}$. All norms considered in
this paper are Euclidean. We let $\log_+ x \eqdef \log x + 1$. A differentiable
function $f$ is $L$-smooth if for any $x, y \in \mathcal{X}$ we have
$\norm{\nabla f(x)-\nabla f(y)} \leq L \norm{x-y}$.

\textbf{Oracle model.} All algorithms we consider shall access gradients through
one of the two oracles defined below.

\begin{definition}\label{def:det-oracle}
We say that $\mathcal{O}(f)$ is a \textbf{deterministic first-order oracle} for the
function $f$ if given a point $x \in \mathcal{X}$ the oracle returns the pair
$\left\{ f(x), \nabla f(x) \right\}$.
\end{definition}

If we only allow the algorithm access to stochastic gradients, then we call this
a stochastic oracle. Our main lower bounds are developed under the following
noise model.

\begin{assumption}\label{asm:almost-sure-bound}
The stochastic gradient estimates are bounded almost surely. That is, there exists some
$R \geq 0$ such that for all $x \in \mathcal{X}$
\begin{align*}
\norm{\hat{g}(x) - \nabla f(x)} \leq R.
\end{align*}
\end{assumption}

The stochastic first-oracle that

\begin{definition}\label{def:stochastic-oracle} We say that $\mathcal{O} (f, R_f)$ is a
\textbf{stochastic first-order oracle} for the function $f$ \textbf{with bound $\sigma_f$} if,
given a point $x \in \mathcal{X}$, it returns a pair of random variables
$\left[ \hat{f} (x), \hat{g} (x) \right]$ such that (a) the estimates are
unbiased $\mathbb{E}[\hat{f}(x)] = f(x)$, $\mathbf{E}[\hat{g}(x)] = \nabla f(x)$, and
(b) the stochastic gradients satisfy \Cref{asm:almost-sure-bound} with $R=\sigma_f$.
\end{definition}

The above oracle restricts the noise to be bounded almost surely. We shall
develop our lower bounds under that oracle. However, for some of the upper
bounds we develop, we shall relax the requirement on the noise from boundedness
to being sub-gaussian (see \Cref{sec:what-noise-distr}), and we shall make this
clear then.

\section{Related Work}

This section reviews existing approaches in the literature aimed at reducing or
eliminating hyperparameter tuning.

\textbf{Parameter-free optimization.} An algorithm $\mathcal{A}$ is called
\emph{parameter-free} if it achieves the convergence rate
$\tilde{\mathcal{O}} \left( \frac{G \norm{x_0 - x_{\ast}}}{\sqrt{T}} \right)$ given $T$
stochastic gradient accesses for any convex function $f$ with stochastic
subgradients bounded in norm by $G$, with possible knowledge of $G$
\citep[Remark 1]{orabona23_normal_gradien_all}. There exists a vast literature
on such methods, particularly in the setting of online learning, see
e.g.~\citep{orabona2020parameter}. Parameter-free optimization differs from
tuning-free optimization in two ways: (a) the $\tilde{\mathcal{O}} (\cdot)$ can suppress
higher-order terms that are not permitted according to the tuning-free
definition, and (b) gives the algorithm possible knowledge of a parameter like
$G$ whereas tuning-free algorithms can only get to see upper and lower bounds on
$G$. Nevertheless, many parameter-free methods do not need any knowledge of
$G$~\citep{cutkosky19_artif_const_lipsc_hints_uncon_onlin_learn,mhammedi19_lipsc_adapt_with_multip_learn,mhammedi20_lipsc_compar_norm_adapt_onlin_learn}.
However,
\citet{cutkosky17_onlin_learn_without_prior_infor,cutkosky17_onlin_convex_optim_with_uncon_domain_losses}
give lower bounds showing that any online learning algorithm insisting on a
linear dependence on $\norm{x_0 - x_{\ast}}$ (as in optimally tuned SGD) must
suffer regret from potentially exponential regret. If we do not insist on a
linear dependence on $\norm{x_0 - x_{\ast}}$, then the best achievable convergence
bound scales $\norm{x_0 - x_{\ast}}^3$, and this is
tight~\citep{mhammedi20_lipsc_compar_norm_adapt_onlin_learn}. None of the
aforementioned lower bounds apply to the setting of stochastic optimization,
since in general online learning assumes an adversarial oracle, which is
stronger than a stochastic oracle.

\textbf{Tuning-free algorithms in the deterministic setting.} Gradient descent augmented with line
search \citep{nesterov14_universal_grad_methods,beck17_first_order_methods_opt}
is tuning-free for smooth convex and nonconvex optimization. Bisection
search \citep{carmon22_makin_sgd_param_free} is tuning-free for both convex and
smooth as well as convex and Lipschitz optimization, as is a restarted version of gradient
descent with Polyak stepsizes~\citep{hazan19_revis_polyak_step_size}. In the
smooth setting, the adaptive descent method
of~\citep{malitsky19_adapt_gradien_descen_without_descen} is also tuning-free.
There are also accelerated methods~\citep{lan23_optim_param_free_gradien_minim},
methods for the Lipschitz
setting~\citep{defazio23_learn_rate_free_learn_by_d_adapt}, methods based on
online learning~\citep{orabona23_normal_gradien_all}, and others.

\textbf{Algorithms for the stochastic setting.} Observe that because online learning is a more
general setting than the stochastic one, we can apply algorithms from online
convex optimization here, like
e.g. \citep{mhammedi20_lipsc_compar_norm_adapt_onlin_learn} coupled with an
appropriate online-to-batch conversion~\citep{hazan22_ocobook}. In more recent
work \citep{carmon22_makin_sgd_param_free,ivgi23_dog_is_sgds_best_frien}, we see
algorithmic developments specific to the stochastic setting. We discuss the
convergence rates these algorithms achieve in more detail in
\Cref{sec:imposs-results-stoch}.

\textbf{Other hyperparameter tuning approaches.} In practice, hyperparameters
are often found by grid search, random search, or methods based on Bayesian
optimization~\citep{Bischl2023}; None of these approaches come with efficient
theoretical guarantees. Another approach is ``meta-optimization'' where we have
a sequence of optimization problems and seek to minimize the cumulative error
over this sequence. Often, another optimization algorithm is then used to select
the learning rates, e.g. hypergradient
descent~\citep{baydin17_onlin_learn_rate_adapt_with_hyper_descen}.
Meta-optimization approaches are quite difficult to establish theoretical
guarantees for, and only recently have some theoretical results been
shown~\citep{chen23_nonst_contr_approac_to_optim}. Our setting in this paper is
different, since rather than seek to minimize regret over a sequence of
optimization problems, we have a single function and an oracle that gives us
(stochastic) gradient estimates for this function.

\textbf{Concurrent work.} In concurrent work, \citet{carmon24_price_adapt_stoch_convex_optim} and \citet{attia24_how_free_is_param_free_stoch_optim} also study lower bounds for first-order stochastic optimization. In both papers, like in our work, the algorithm is provided with a certain range that the problem parameters fall in (what we term as \emph{hints}) and must make use of only that to minimize the function with stochastic gradient evaluations. \citet{carmon24_price_adapt_stoch_convex_optim} study what is the minimum possible multiplicative factor slowdown any algorithm must suffer compared to optimally-tuned baselines when provided access only to hints, which they term the \emph{price of adaptivity}. They provide lower bounds for stochastic convex optimization for Lipschitz functions in expectation and with high probability, and also consider the case where some of the problem parameters have no uncertainty (e.g. when we know the Lipschitz constant but not the initial distance to the optimum). Our lower bound in this setting~(\Cref{thm:impossibility-result-non-smooth}) rules out any polylogarithmic price of adaptivity as impossible. Additionally, we also give lower bounds for nonconvex and smooth convex optimization~(\Cref{thm:impossibility-result-smooth,thm:impossibility-result-nonconvex}).

\citet{attia24_how_free_is_param_free_stoch_optim} study stochastic optimization in a similar setting to ours, and give a new upper bound for restarted non-convex SGD that achieves a similar convergence guarantee to \Cref{thm:restarted-sgd}. We give our upper bound under a slightly more general noise distribution (that the noise has subgaussian norm) at the cost of a polylogarithmic dependence on the problem dimension. We also additionally give a lower bound that rules out the stronger in-expectation convergence guarantee for nonconvex SGD. In the convex setting, \citet{attia24_how_free_is_param_free_stoch_optim} give lower bounds for smooth and nonsmooth stochastic optimization that show a polynomial dependence on the hints is the best we can hope to achieve, and give a matching upper bound based on restarted SGD with AdaGrad-like stepsizes. In contrast, for our upper bounds in this case we study more specifically which noise distributions are amenable to optimization and prove results for the DoG and DoWG algorithms (with no restarting procedures). Additionally, we also investigate whether tuning-free optimization is possible under a bounded domain and provide guarantees for DoG/DoWG there (\Cref{thm:dog-dowg-tuning-free-bounded}).

\section{Tuning-Free Optimization Under a Bounded Domain}

We begin our investigation by studying the \textbf{bounded setting}, where we
make the following assumption on the minimization problem~\eqref{eq:opt-problem}:

\begin{assumption}\label{asm:bounded-domain}
The optimization domain $\mathcal{X}$ is bounded. There exists some constant $D > 0$ such
that $\norm{x-y} \leq D$ for all $x, y \in \mathcal{X}$.
\end{assumption}

We seek a tuning-free version of SGD. Recall that SGD achieves with probability at least
$1-\delta$ the following convergence guarantee~\citep{jain19_makin_last_iterat_sgd_infor_theor_optim,liu23_high_probab_conver_stoch_gradien_method}

\begin{align}\label{eq:53}
f(x_{\mathrm{out}}) - f_{\ast} \leq \upsilon \cdot \begin{cases}
                              \frac{D L^2}{T} + \frac{\sigma D}{\sqrt{T}}  & \text { if } f \text { is $L$-smooth}, \\
                              \frac{\sqrt{G^2 + \sigma^2} D}{\sqrt{T}}   & \text { if }  f \text { is $G$-Lipschitz},
                              \end{cases}
\end{align}
where $\nu = \mathrm{poly\log} \frac{1}{\delta}$ and $\sigma$ is an upper bound on the
stochastic gradient noise (per \Cref{asm:almost-sure-bound}). To achieve the
convergence guarantee given by~\cref{eq:53}, we need to know the parameters
$D, \sigma$, and $L$ in the smooth case or $G$ in the nonsmooth case. Per
\Cref{def:tuning-free-general}, a tuning-free version of SGD will thus be given
the hints $D \in [\underline{D}, \overline{D}]$,
$\sigma \in [\underline{\sigma}, \overline{\sigma}]$, and either
$L \in [\underline{L}, \overline{L}]$ in the smooth setting or
$G \in [\underline{G}, \overline{G}]$ in the nonsmooth setting. Given those hints,
we then ask the algorithm to achieve the same rate as \cref{eq:53} up to a
multiplicative polylogarithmic function of the hints.

It turns out that tuning-free optimization under a bounded domain is solvable in
many ways. Many methods from the online learning literature, e.g.~\citep{cutkosky17_bounded_no_params,mhammedi19_lipsc_adapt_with_multip_learn,cutkosky19_artif_const_lipsc_hints_uncon_onlin_learn}
can solve this problem when combined with standard online-to-batch conversion bounds. We
give the details for this construction for one such algorithm in the next proposition:

\begin{restatable}{proposition}{cutkoskybounded}\label{prop:cutkosky-alg}
Coin betting through Online Newton Steps with Hints \citep[Algorithm
1]{cutkosky19_artif_const_lipsc_hints_uncon_onlin_learn} is tuning-free in the
bounded setting.
\end{restatable}
The proof of this result is provided in the appendix, and essentially just
combines \citep[Theorem 2]{cutkosky19_artif_const_lipsc_hints_uncon_onlin_learn}
with online-to-batch conversion.

In this paper, we shall focus particularly on methods that fit the stochastic
gradient descent paradigm, i.e. that use updates of the form
$x_{k+1} = x_k - \eta_k g_k$, where $g_k$ is the stochastic gradient at
step $k$. Two methods that fit this paradigm are DoG~\citep{ivgi23_dog_is_sgds_best_frien}
and DoWG~\citep{khaled23_dowg_unleas}. DoG uses stepsizes of the form
\begin{align}
\label{eq:dog-update} \eta_t = \frac{\overline{r}_t}{\sqrt{u_t}}, && \overline{r}_t = \max_{k \leq t} (\norm{x_t - x_0}, r_{\epsilon}), && u_t = \sum_{k=0}^t \sqn{g_k},
\end{align}
where $r_{\epsilon}$ is a parameter that we will always set to $\underline{D}$. Similarly, DoWG
uses stepsizes of the form
\begin{align}
\label{eq:dowg-update} \eta_t = \frac{\overline{r}_t^2}{\sqrt{v_t}}, && \overline{r}_t = \max_{k \leq t} (\norm{x_t - x_0}, r_{\epsilon}), && v_t = \sum_{k=0}^t \overline{r}_k^2 \sqn{g_k}.
\end{align}

The next theorem shows that in the bounded setting, both DoG and DoWG are
tuning-free.

\begin{restatable}{theorem}{dogdowgbounded}\label{thm:dog-dowg-tuning-free-bounded}
DoG and DoWG are tuning-free in the bounded setting. That is, there exists some
$\iota = \mathrm{poly\log} (\frac{\overline{D}}{\underline{D}}, \frac{\overline{\sigma}}{\underline{\sigma}}, T, \delta^{-1})$
such that
\begin{align*}
f(x_{\mathrm{out}}) - f_{\ast} \leq \iota \cdot \begin{cases}
                              \frac{D L^2}{T} + \frac{\sigma D}{\sqrt{T}}  & \text { if } f \text { is $L$-smooth}, \\
                              \frac{\sqrt{G^2 + \sigma^2} D}{\sqrt{T}}   & \text { if }  f \text { is $G$-Lipschitz}.
                              \end{cases}
\end{align*}
This rate is achieved simultaneously for both classes of functions without prior
knowledge of whether $f$ is smooth or Lipschitz (and thus no usage of the hints
$\underline{L}, \overline{L}, \underline{G}, \overline{G}$).
\end{restatable}

This theorem essentially comes for free by modifying the results in
\citep{ivgi23_dog_is_sgds_best_frien,khaled23_dowg_unleas}, and while the proof
modifications are quite lengthy we claim no significant novelty here. We note
further that unlike \citep[Algorithm
1]{cutkosky19_artif_const_lipsc_hints_uncon_onlin_learn}, both DoG and DoWG are
\emph{single-loop} algorithms-- they do not restart the optimization process or
throw away progress. This is a valuable property and one of the reasons we focus
on these algorithms in the paper. Moreover, DoG and DoWG are
\emph{universal}. An algorithm is universal if it achieves the same rate as SGD
for Lipschitz functions and also takes advantage of smoothness when it
exists~\citep{levy17_onlin_to_offlin_conver_univer}, without any prior knowledge
of whether $f$ is smooth. DoG and DoWG enjoy this property in the bounded domain setting.

\section{Tuning-free Optimization Under an Unbounded Domain}
\label{sec:tuning-free-optim-unbounded}

We now continue our investigation to the general, unbounded setting where
$\mathcal{X} = \mathbb{R}^d$. Now, the diameter $D$ in \Cref{asm:bounded-domain} is infinite. The
convergence of SGD is then characterized by the initial distance to the optimum
$D_{\ast} = \norm{x_0 - x_{\ast}}$~\citep{liu23_high_probab_conver_stoch_gradien_method}.
We can show that SGD with optimally-tuned stepsizes achieves with probability at
least $1-\delta$ the convergence rates
\begin{align}\label{eq:52}
f(x_{\mathrm{out}}) - f_{\ast} \leq \upsilon \cdot \begin{cases}
                              \frac{D_{\ast} L^2}{T} + \frac{\sigma D_{\ast}}{\sqrt{T}}  & \text { if } f \text { is $L$-smooth}, \\
                              \frac{\sqrt{G^2 + \sigma^2} D_{\ast}}{\sqrt{T}}   & \text { if }  f \text { is $G$-Lipschitz},
                              \end{cases}
\end{align}
where $\upsilon = \mathrm{poly}\log \frac{1}{\delta}$, $\sigma$ is the maximum stochastic
gradient noise norm, and $D_{\ast} = \norm{x_0 - x_{\ast}}$ is the initial distance
from the minimizer. An algorithm is a tuning-free version of SGD in the unbounded setting
if it can match the best SGD rates given by \Cref{eq:52} up to polylogarithmic
factors given access to the hints
$\underline{D}, \overline{D}, \underline{\sigma}, \overline{\sigma}$, and
$\underline{G}, \overline{G}$ or $\underline{L}, \overline{L}$. This is a tall
order: unlike in the bounded setting, a tuning-free algorithm now has to compete
with SGD's convergence on \emph{any} possible initialization.

\textbf{Deterministic setting.} When there is no stochastic gradient noise, i.e.
$\sigma = 0$ and the algorithm accesses gradients according to the deterministic
first-order oracle (\Cref{def:det-oracle}), Tuning-free versions of gradient
descent exist. For example, the Adaptive Polyak
algorithm~\citep{hazan19_revis_polyak_step_size}, a restarted version of
gradient descent with the Polyak stepsizes~\citep{polyak1987introduction} is
tuning-free:

\begin{restatable}[\citet{hazan19_revis_polyak_step_size}]{proposition}{adaptivepolyak}\label{prop:adaptive-polyak-deterministic}
The Adaptive Polyak algorithm from \citep{hazan19_revis_polyak_step_size} is tuning-free in the deterministic setting.
\end{restatable} 

This is far from the only solution, and we mention a few others next.
Parameter-free methods augmented with normalization are also tuning-free and
universal, e.g.\ plugging in $d_0 = \underline{D}$
in~\citep{orabona23_normal_gradien_all} gives tuning-free algorithms matching
SGD. The bisection algorithm from \citep{carmon22_makin_sgd_param_free} is also
tuning-free, as is the simple doubling trick. Finally, T-DoG and T-DoWG,
variants of DoG and DoWG which use polylogarithmically smaller stepsizes than
DoG and DoWG, are also tuning-free, as the following direct corollary of
\citep{ivgi23_dog_is_sgds_best_frien,khaled23_dowg_unleas} shows.

\begin{restatable}{proposition}{dogdowgdetermenistic}\label{prop:dog-dowg-deterministic}
T-DoG and T-DoWG are tuning-free in the deterministic setting.
\end{restatable}

T-DoG and T-DoWG use the same stepsize structure as DoG and DoWG (given
in~\cref{eq:dog-update,eq:dowg-update}), but divide these stepsizes by running
logarithmic factors as follows
\begin{align*}
  \text{T-DoG:} \ \eta_t = \frac{\overline{r}_t}{\sqrt{u_t} \log_+ \frac{u_t}{u_0}},
&& \text{T-DoWG:} \ \eta_t = \frac{\overline{r}_t^2}{\sqrt{v_t} \log_+ \frac{v_t}{v_0}}.
\end{align*}
Both methods achieve the same convergence guarantee as in~\cref{eq:52} up to
polylogarithmic factors in the hints.

\subsection{Impossibility Results in the Stochastic Setting}\label{sec:imposs-results-stoch}
The positive results in the deterministic setting give us some hope to obtain a
tuning-free algorithm. Unfortunately, the stochastic setting turns out to be a
tougher nut to crack. Our first major result, given below, slashes any hope of
finding a tuning-free algorithm for smooth and stochastic convex optimization.

\begin{theorem}\label{thm:impossibility-result-smooth}
For any polylogarithmic function $\iota: \mathbb{R}^4 \rightarrow \mathbb{R}$ and any algorithm $\mathcal{A}$, there
exists a time horizon $T$, an $L$-smooth and convex function $f$, and a
stochastic oracle $\mathcal{O}(f, \sigma_f)$, and valid hints
$\underline{L}, \overline{L}, \underline{D}, \overline{D}, \underline{\sigma}, \overline{\sigma}$
such that the algorithm $\mathcal{A}$ initialized at some $x_0$ returns with some constant probability a point
$x_{\mathrm{out}}$ satisfying
\begin{align*}
\mathrm{Error}_{\mathcal{A}} &= f(x_{\mathrm{out}}) - f_{\ast} > \iota \left( \frac{\overline{L}}{\underline{L}}, \frac{\overline{D}}{\underline{D}}, \frac{\overline{\sigma}}{\underline{\sigma}}, T \right) \cdot \left[ \frac{L D_{\ast}^2}{T} + \frac{\sigma_f D_{\ast}}{\sqrt{T}} \right],
\end{align*}
where $D_{\ast} = \norm{x_0 - x_{\ast}}$ is the initial distance to the optimum and $\sigma_f$
is the maximum norm of the stochastic gradient noise.
\end{theorem}

\textbf{Proof idea.} This lower bound is achieved by $1$-dimensional functions.
In particular, we construct two one-dimensional quadratic functions $f$ and $h$
with associated oracles $\mathcal{O}(f, \sigma_f)$ and $\mathcal{O}(h, \sigma_h)$, and we supply the algorithm
with hints that are valid for both functions and oracles. We show that with some constant
probability, the algorithm observes the same stochastic gradients from both
$\mathcal{O}(f, \sigma_f)$ and $\mathcal{O}(h, \sigma_h)$ for the entire run. Since the algorithm cannot tell
apart either oracle, it must guarantee that
\cref{eq:52} holds with high probability for both $f$
and $h$ if it is to be tuning-free. Now, if we choose $f$ and $h$ further apart,
ensuring that their respective oracles return the same gradients with some
constant probability becomes harder. On the other hand, if we choose $f$ and $h$
too close, the algorithm can conceivably guarantee that
\cref{eq:52} holds (up to the same polylogarithmic factor of the hints) for both
of them. By carefully choosing $f$ and $h$ to balance out this tradeoff, we show
that no algorithm can be tuning-free in the unbounded and stochastic setting.
The full proof is provided in \Cref{sec:proof-impossibility-smooth} in the
appendix.

\textbf{Comparison with prior lower bounds.} The above theorem shows a
fundamental separation between the deterministic and stochastic settings when
not given knowledge of the problem parameters. The classical lower bounds for
deterministic and stochastic optimization
algorithms~\citep{nesterov18_lectures_cvx_opt,woodworth16_tight_compl_bound_optim_compos_objec,carmon17_lower_bound_findin_station_point_ii}
rely on a chain construction that is agnostic to whether the optimization
algorithm has access to problem parameters. On the other hand, lower bounds from
the online learning literature show that tuning-free optimization matching SGD
is impossible when the oracle can be adversarial (and not stochastic), see e.g.\
\citep{cutkosky17_onlin_learn_without_prior_infor,cutkosky17_onlin_convex_optim_with_uncon_domain_losses}.
However, adversarial oracles are much stronger than stochastic oracles, as they
can change the function being optimized in response to the algorithm's choices.
Our lower bound is closest in spirit to the lower bounds from the
stochastic multi-armed bandits literature that also rely on confusing the
algorithm with two close problems~\citep{mannor04_bandits}.

Our next result shows that tuning-free optimization is also impossible in the
nonsmooth case.

\begin{theorem}\label{thm:impossibility-result-non-smooth}
For any polylogarithmic function $\iota: \mathbb{R}^4 \rightarrow \mathbb{R}$ and any
algorithm $\mathcal{A}$, there exists a time horizon $T$, valid hints
$\underline{L}, \overline{L}, \underline{D}, \overline{D}, \underline{\sigma}, \overline{\sigma}$,
an $G$-Lipschitz and convex function $f$ and an oracle $\mathcal{O}(f, \sigma_f)$ such that the
algorithm $\mathcal{A}$ returns with some constant probability a point $x_{\mathrm{out}}$ satisfying
\begin{align*}
  \mathrm{Error}_{\mathcal{A}} &= f(x_{\mathrm{out}}) - f_{\ast} > \iota \left( \frac{\overline{G}}{\underline{G}}, \frac{\overline{D}}{\underline{D}}, \frac{\overline{\sigma}}{\underline{\sigma}}, T \right) \cdot \left[ \frac{\sqrt{G^2 + \sigma^2} D_{\ast}}{\sqrt{T}}  \right].
\end{align*}
\end{theorem}

The proof technique used for this result relies on a similar construction as
\Cref{thm:impossibility-result-smooth} but uses the absolute loss instead of
quadratics.

\textbf{Existing algorithms and upper bounds in the stochastic setting.}
\citet{carmon22_makin_sgd_param_free}
give a restarted variant of SGD with bisection search. If $f$ is $G$-Lipschitz
and all the stochastic gradients are also bounded by $G$, their method uses the
hint $\overline{G} \geq G$ and achieves the following convergence rate with
probability at least $1-\delta$
\begin{align}
f(\hat{x}) - f_{\ast} \leq c \iota \left(  \eta_{\epsilon} \left (G + \frac{\iota \overline{G}}{T} \right) + \frac{D_{\ast} G}{\sqrt{T}} + \frac{D_{\ast} \overline{G}}{T}   \right),
\end{align}
where $c$ is an absolute constant, $\iota$ a poly-logarithmic and double-logarithmic factor,
and $\eta_{\epsilon}$ an input parameter. If we set $\eta_{\epsilon} = \underline{D}/T \leq \frac{D_{\ast}}{T}$, then the
guarantee of this method becomes
\begin{align}
\label{eq:50}
f(\hat{x}) - f_{\ast} \leq c \iota \left( \frac{D_{\ast} G}{\sqrt{T}} + \frac{D_{\ast} \overline{G}}{T}  \right).
\end{align}
Unfortunately, this dependence does not meet our bar  as the polynomial
dependence on $\overline{G}$ is in a higher-order term. A similar result is
achieved by DoG~\citep{ivgi23_dog_is_sgds_best_frien}. A different sort of guarantee is achieved by
\citet{mhammedi20_lipsc_compar_norm_adapt_onlin_learn}, who give a method with regret
\begin{align}
\label{eq:41}
  \frac{1}{T} \sum_{t=0}^{T-1} \ev{g_t, x_t - x_{\ast}} \leq c \iota \left( \frac{G D_{\ast}}{\sqrt{T}} + \frac{G D_{\ast}^3}{T}  \right),
\end{align}
for some absolute constant $c$ and polylogarithmic factor $\iota$. This result is in
the adversarial setting of online learning, and clearly does not meet the bar
for tuning-free optimization matching SGD due to the cubic term $D_{\ast}^3$.

\subsection{Guarantees Under Benign Noise}
\label{sec:what-noise-distr}

In the last subsection, we saw that tuning-free optimization is in general
impossible. However, it is clear that \emph{sometimes} it is possible to get
within the performance of tuned SGD with self-tuning
methods~\citep{ivgi23_dog_is_sgds_best_frien,defazio23_learn_rate_free_learn_by_d_adapt}.
However, the oracles used
in \Cref{thm:impossibility-result-smooth,thm:impossibility-result-non-smooth} provide
stochastic gradients $g(x)$ such that the noise is almost surely bounded (i.e.
satisfies \Cref{asm:almost-sure-bound}):
\begin{align*}
\norm{g(x) - \nabla f(x)} \leq R.
\end{align*}
So boundedness is clearly not enough to enable tuning-free optimization.
However, we know from prior results (e.g. \citep{carmon22_makin_sgd_param_free})
that if we can reliably estimate the upper bound $R$ on the noise, we can adapt to unknown
distance to the optimum $D_{\ast}$ or the smoothness constant $L$. The main issue that
the oracles in the lower bound of \Cref{thm:impossibility-result-smooth} make it
impossible to do that: while the noise $n(x) = g(x) - \nabla f(x)$ is bounded almost
surely by $R$, the algorithm only gets to observe the same noise $n(x)$ for the
entire optimization run. This foils any attempt at estimating $R$ from the
observed trajectory.

\textbf{A note on notation in this section and the next.} In the past section we
used $\sigma$ to denote a uniform upper bound on the gradient noise, while in this
section and the next we use $\sigma$ to denote the \emph{variance} of the stochastic
gradient noise $n(x)$ rather than a uniform upper bound on it. Instead, we use
$R$ to denote the uniform upper bound on the noise.

We will see that for some notion of \emph{benign noise}, tuning-free
optimization matching SGD is possible. We will develop our results under a more
general assumption on the distribution of the stochastic gradient noise
$g(x) - \nabla f(x)$

\begin{assumption}\label{asm:subg-noise}
(Noise with Sub-Gaussian norm).
For all $x \in \mathbb{R}^d$, the noise vector $n(x) = g(x) - \nabla f(x)$ satisfies
\begin{itemize}
\item $n(x)$ is unbiased: $\ec{g(x)} = \nabla f(x)$.
\item $n(x)$ has sub-gaussian norm with modulus $R$:
\begin{align*}
\pr[ \norm{n(x)} \geq t ] \leq 2 \exp \left( \frac{-t^2}{2 R^2}  \right).
\end{align*}
\item $n(x)$ has bounded variance: $\ecn{n(x)} = \sigma^2 < +\infty$.
\end{itemize}
\end{assumption}

This assumption is very general, it subsumes bounded noise (where $R = \sigma_f$) and
sub-gaussian noise. The next definition gives a notion of \emph{signal-to-noise} that
turns out to be key in characterizing benign noise distributions.

\begin{definition}\label{def:signal-to-noise}
Suppose the stochastic gradient noise satisfies \Cref{asm:subg-noise}. We define the \textbf{signal-to-noise}
ratio associated with the noise as
\begin{align*}
K_{\mathrm{snr}} &= \frac{\sigma}{R} \leq 1.
\end{align*}
\end{definition}

To better understand the meaning of $K_{\mathrm{snr}}$, we consider the
following example. Let $Y$ be a random vector with mean $\ec{Y} = \mu$ and
variance $\ecn{Y} = \sigma^2$. Suppose further that the errors $\norm{Y-\mu}$ are
bounded almost surely by some $R$. Then $Y-\mu$ satisfies the assumptions in
\Cref{asm:subg-noise}. Let $Y_1, \ldots, Y_b$ be independent copies of $Y$. Through
standard concentration results (see \Cref{lem:convergence-of-sample-variance})
we can show that with high probability and for large enough $b$
\begin{align*} \hat{\sigma} \eqdef \frac{1}{n} \sum_{i=1}^n \sqn{Y-\mu} \approxeq \sigma^2.
\end{align*} Now observe that if the ratio $K_{\mathrm{snr}}$ is small, then we
cannot use the sample variance $\hat{\sigma}$ as an estimator for the almost-sure
bound $R$. But if the ratio $K_{\mathrm{snr}}$ is closer to $1$, then we have
$\sigma^2 \approxeq R^2$ and we can use $\hat{\sigma}$ as an estimator for $R$. This fixes the
problem we highlighted earlier: now we are able to get an accurate estimate of
$R$ from the observed stochastic gradients. The next proposition gives examples
of noise distributions where $K_{\mathrm{snr}}$ is close to $1$.

\begin{proposition}\label{prop:special-noise}
Suppose that the noise vectors $g (x) - \nabla f(x)$ follow one of the following two
distributions:
\begin{itemize}
\item A Gaussian distribution with mean $0$ and covariance
        $\frac{\sigma^2}{d} \cdot I_{d \times d}$, with $\sigma > 0$.
\item A Bernoulli distribution, where $[g(x) - \nabla f(x)] = \pm \sigma \phi(x)$
        with equal probability for some $\phi$ such that $\norm{\phi(x)}_2 = 1$ almost
        surely.
\end{itemize}
Then $K_{\mathrm{snr}} = \mathcal{O}(1)$.
\end{proposition}

We now give an algorithm whose convergence rate characterized by the
signal-to-noise ratio $K_{\mathrm{snr}}$. We combine a variane estimation
procedure with the T-DoG/T-DoWG algorithms in
\Cref{alg:dog-with-estimation,alg:dowg-with-estimation}. The next theorem gives
the convergence of this algorithm. This theorem is generic, and does not
guarantee any tuning-free matching of SGD, but can lead to tuning-free matching
of SGD if the signal to noise ratio is high enough.

\begin{theorem}\label{thm:dog-dowg-special-noise}
Suppose we are given access to stochastic gradient estimates $g(x)$ such that
the noise vectors $[g(x) - \nabla f(x)] \in \mathbb{R}^d$ satisfy \Cref{asm:subg-noise} with
modulus $R$ and signal-to-noise ratio $K_{\mathrm{snr}}$. If we run T-DoG or
T-DoWG with variance estimation
(\Cref{alg:dog-with-estimation,alg:dowg-with-estimation}) with a minibatch size
$b \geq 2$ large enough to satisfy
\begin{align*} c \cdot \left[ \sqrt{\frac{\log \frac{2b T}{\delta}}{b}} + \frac{\log \frac{2(b \lor d) T}{\delta}}{b} \right] \leq K_{\mathrm{snr}}^2 - \theta,
\end{align*} where $c$ is some absolute constant and $\theta \in [0, K_{\mathrm{snr}}]$
is some known constant. Then
\Cref{alg:dog-with-estimation,alg:dowg-with-estimation} with either option
returns a point $x_{\mathrm{out}}$ such that with probability at least $1-\delta$,
\begin{itemize}
\item If $f$ is $L$-smooth:
\begin{align*} f(x_{\mathrm{out}}) - f_{\ast} \leq c \iota \left( \frac{L D_{\ast}^2 b}{T_{\mathrm{total}}} + \frac{\theta^{-1} R D_{\ast} \sqrt{b}}{\sqrt{T_{\mathrm{total}}}} \right),
\end{align*} where $D_{\ast} = \norm{x_0 - x_{\ast}}$, $c$ is an absolute constant,
$\iota$ is a polylogarithmic factor of the hints, and $T_{\mathrm{total}}$ denotes
the total number of stochastic gradient accesses.
\item If $f$ is $G$-Lipschitz:
\begin{align*} f(x_{\mathrm{out}}) - f_{\ast} \leq c \iota \frac{\sqrt{G^2 + \theta^{-2} R^2} D_{\ast} \sqrt{b}}{\sqrt{T_{\mathrm{total}}}}.
\end{align*}
\end{itemize}
\end{theorem}

\textbf{Note on dimension dependence in \Cref{thm:dog-dowg-special-noise}.} We
note that the logarithmic dimension dependence on the dimension $d$ can be
removed if, rather than assuming the norm of the noise is subgaussian, we
assumed it was bounded.

On the surface, it looks like \Cref{thm:dog-dowg-special-noise} simply trades
off knowledge of the absolute bound on the noise $R$ with knowledge of some
constant $\theta$ that lies in the interval $[0, K_{\mathrm{snr}}]$. In order to see
how \Cref{thm:dog-dowg-special-noise} can be useful, consider the special cases
given in \Cref{prop:special-noise}. For these noise distributions, we see that
choosing a minibatch size $b \approx \mathcal{O} (\log \frac{2 d T}{\delta} + 1)$ suffices to ensure
\Cref{alg:dog-with-estimation,alg:dowg-with-estimation} converges with the
simple choice $\theta=\frac{1}{2}$. Even though we had no apriori knowledge of the variance
$\sigma^2$ and did not assume the noise distribution was stationary, we could still
optimize the function. In general, the minibatch size
$b \approx \mathcal{O} (\log \frac{2 d T}{\delta} + 1)$ suffices as long as $K_{\mathrm{snr}}$ is
bounded away from zero by some constant. The final cost of running the algorithm
is $T_{\mathrm{total}} = b \cdot T = \iota T$, where $\iota$ is some polylogarithmic factor.
Therefore, we only pay a logarithmic price for not knowing the distribution. Of
course, if $K_{\mathrm{snr}}$ is small enough, there can be no optimization--
there is not enough signal to do any estimation of the sub-gaussian modulus $R$.
The distribution used in \Cref{thm:impossibility-result-smooth} does force
$K_{\mathrm{snr}} \leq \frac{1}{T}$.

\section{Nonconvex Tuning-Free Optimization}
\label{sec:nonconvex-tuning-free}

In this section, we consider the case where the optimization
problem~\eqref{eq:opt-problem} is possibly nonconvex. Throughout the section, we
assume that $f$ is $L$-smooth and lower bounded by some $f_{\ast} \in \mathbb{R}$. In this
setting, SGD with a tuned stepsize achieves the following rate in expectation
\begin{align}
\label{eq:48}
\begin{split}
  \frac{1}{T} \sum_{t=0}^{T-1} \sqn{\nabla f(x_t)} \leq c \left[ \sqrt{\frac{L (f(x_0) - f_{\ast}) \sigma^2}{T}} + \frac{L (f(x_0) - f_{\ast})}{T}  \right],
\end{split}
\end{align}
for some absolute constant $c > 0$. This rate is known to be tight for
convergence in
expectation~\citep{arjevani19_lower_bound_non_convex_stoch_optim}. However, it
is not known if it is tight for returning a \emph{high probability} guarantee.
The best-known high-probability convergence rate for SGD is given by
\citep[Theorem 4.1]{liu23_high_probab_conver_stoch_gradien_method} and
guarantees with probability at least $1-\delta$ that
\begin{align}\label{eq:40}
\begin{split}
\frac{1}{T} \sum_{t=0}^{T-1} \sqn{\nabla f(x_t)} \leq 5 \sqrt{\frac{L (f(x_0) - f_{\ast}) R^2}{T}} + \frac{2 (f(x_0) - f_{\ast}) L}{T}  + \frac{12 R^2 \log \frac{1}{\delta}}{T}.
\end{split}
\end{align}

We now consider tuning-free algorithms that can match the performance of SGD
characterized by either~\cref{eq:48} or~\cref{eq:49}.
Per~\Cref{def:tuning-free-general}, an algorithm $\mathcal{B}$ is given (1) an
initialization $x_0$, (2) a budget of $T_{\mathrm{total}}$ stochastic gradient
accesses, and (3) hints
$\underline{L}, \overline{L}, \underline{R}, \overline{R}, \underline{\Delta}, \overline{\Delta}$
on the problem parameters such that (a) if $L$ is the smoothness constant of $f$
then $L \in [\underline{L}, \overline{L}]$, (b)
$R \in [\underline{R}, \overline{R}]$, and (c)
$\Delta \eqdef f(x_0) - f_{\ast} \in [\underline{\Delta}, \overline{\Delta}]$. We call $\mathcal{B}$
\textbf{strongly} tuning-free if it matches the performance of SGD characterized
by~\cref{eq:48} up to polylogarithmic factors. Alternatively, if it instead matches the
weaker guarantee given by~\cref{eq:40} then we call it \textbf{weakly}
tuning-free.

Our first result in this setting shows that we cannot hope to achieve the rate
given by~\cref{eq:48} in high probability, even given access to hints on all the
problem parameters.

\begin{theorem}\label{thm:impossibility-result-nonconvex}
For any polylogarithmic function $\iota: \mathbb{R}^4 \rightarrow \mathbb{R}$ and any
algorithm $\mathcal{A}$, there exists a time horizon $T$, valid hints
$\underline{L}, \overline{L}, \underline{\Delta}, \overline{\Delta}, \underline{\sigma}, \overline{\sigma}$,
an $L$-smooth and lower-bounded function $f$ and an oracle $\mathcal{O}(f, \sigma_f)$ such that the
algorithm $\mathcal{A}$ returns with some constant probability a point $x_{\mathrm{out}}$ satisfying
\begin{align*}
  \mathrm{Error}_{\mathcal{A}} &= \sqn{\nabla f(x_{\mathrm{out}})} > \iota \left( \frac{\overline{L}}{\underline{L}}, \frac{\overline{\Delta}}{\underline{\Delta}}, \frac{\overline{\sigma}}{\underline{\sigma}}, T \right) \cdot \left[ \sqrt{\frac{L \Delta \sigma^2}{T}} + \frac{L \Delta}{T}  \right],
\end{align*}
where $\Delta = f(x_0) - f_{\ast}$
\end{theorem}

\begin{algorithm}[t]
\caption{Restarted SGD}
\label{alg:restarted-sgd}
\begin{algorithmic}[1]
\REQUIRE Initialization $y_0$, probability $\delta$, hints
$\underline{R}, \overline{R}, \underline{\Delta}, \overline{\Delta}, \underline{L}, \overline{L}$,
total budget $T_{\mathrm{total}}$
\STATE Set $\eta_{\epsilon} = \frac{1}{\overline{L}}$ and
\begin{align}
\label{eq:45}
N = 1 + \ceil*{\log \left(  \frac{\min( \overline{L}, \sqrt{\frac{5 T \overline{R}^2}{2 \underline{\Delta}}})}{\max (\underline{L}, \sqrt{\frac{5 T \underline{R}^2}{\overline{\Delta}}})}  \right)}
\end{align}
\IF{$T_{\mathrm{total}} < N$}
\STATE \textbf{Return} $y_0$.
\ENDIF
\STATE Set the per-epoch iteration budget as $T = \ceil{T_{\mathrm{total}} / N}$.
\FOR{$n = 1$ to $N$}
    \STATE $\eta = \eta_{\epsilon} 2^n$
    \STATE Run SGD for $T$ iterations with stepsize $\eta$ starting from $y_0$ to get outputs
    $x^n_1, \ldots, x^n_T$.
    \STATE Set $y_n, \hat{g}_n = \mathrm{FindLeader}(S, \delta, T)$ (see \Cref{alg:find-leader}).
\ENDFOR
\STATE \textbf{Return} $\bar{y} = \arg\min_{n \in [N]} \norm{\hat{g}_n}$.
\end{algorithmic}
\end{algorithm}

Surprisingly, our next theorem shows that the rate given by~\cref{eq:40}
\emph{is} achievable up to polylogarithmic factors given only access to hints.
To achieve this, we use a restarted variant of SGD (Algorithm~\ref{thm:restarted-sgd}) combined with a ``Leader
Finding'' procedure that selects a well-performing iterate by subsampling.

\begin{theorem}\label{thm:restarted-sgd}
(Convergence of Restarted SGD) Let $f$ be an $L$-smooth function lower bounded
by $f_{\ast}$ and suppose the stochastic gradient noise vectors satisfy~\Cref{asm:subg-noise}. Suppose that we are given the following
hints on the problem parameters: (a) $L \in [\underline{L}, \overline{L}]$, (b)
$R_f \in [\underline{R}, \overline{R}]$, and (c)
$\Delta_f \eqdef f(x_0) - f_{\ast} \in [\underline{\Delta}, \overline{\Delta}]$. Then there exists some absolute constant $c$ such
that the output of Algorithm~\ref{alg:restarted-sgd} satisfies after
$T_{\mathrm{total}} \cdot \log_+ \frac{1}{\delta}$ stochastic gradient evaluations
\begin{align*}
\begin{split}
\sqn{\nabla f(\overline{y})} \leq c \cdot \frac{R^2 \log \frac{2d \max(\log \frac{1}{\delta}, N)}{\delta}}{T_{\mathrm{total}}} + c \cdot N \cdot \left[ \sqrt{\frac{L (f(y_0) - f_{\ast}) R^2}{T_{\mathrm{total}}}} + \frac{(f(y_0) - f_{\ast}) L}{T_{\mathrm{total}}}  \right],
\end{split}
\end{align*}
where $c$ is an absolute constant, $N$ is a polylogarithmic function of the
hints defined in \cref{eq:45}, and $d$ is the problem dimensionality.
\end{theorem}

\textbf{Discussion of Theorem~\ref{thm:restarted-sgd}}. This theorem shows that
in the nonconvex setting, we pay only an additional polylogarithmic factor to
achieve the same high-probability rate as when we know all parameters. We
emphasize that we do not know if the rate given by~\cref{eq:40} is tight, but it
is the best in the literature. Finally, the logarithmic dimension dependence on
the dimension $d$ can be removed if, rather than assuming the norm of the noise
is subgaussian, we assumed that it was bounded almost surely.

\textbf{Proof Idea.} The proof is an application of the so-called ``doubling
trick'' with a careful comparison procedure. If we start with a small enough stepsize,
we only need to double a logarithmic number of times until we find a stepsize
$\eta^{\prime}$ such that $\frac{\eta_{\ast}}{2} \leq \eta^{\prime} \leq \eta_{\ast}$, where $\eta_{\ast}$ is the optimal
stepsize for SGD on this problem. We therefore run SGD for $N$ epochs with a
carefully chosen $N$, each time doubling the stepsize. At the end of every SGD
run, we run the $\mathrm{FindLeader}$ procedure (\Cref{alg:find-leader}) to get
with high probability a point $y_n$ such that
\begin{align*}
  \sqn{\nabla f(y_n)} \leq \frac{1}{T} \sum_{t=0}^{T-1} \sqn{\nabla f(x_t^n)},
\end{align*}
where $x_1^n, \ldots, x_T^n$ are the SGD iterates from the $n$-th epoch. Finally, we
know that at least one of these $N$ points $y_1, \ldots, y_N$ has small gradient
norm, so we return the point with the minimal estimated gradient norm and bound
the estimation error as a function of $T$. The total number of gradient accesses
performed is at most $N (T + M T) \approx T_{\mathrm{total}} \cdot \log_+ \frac{1}{\delta}.$
Therefore, both the restarting and comparison procedures add at most a
logarithmic number of gradient accesses.

\textbf{Related work.} Many papers give high probability bounds for SGD or
AdaGrad and their variants in the nonconvex
setting~\citep{ghadimi2013stochastic,madden20_high_probab_conver_bound_non,Lei2021,li18_conver_stoch_gradien_descen_with_adapt_steps,li20_high_probab_analy_adapt_sgd_with_momen,faw22_power_adapt_sgd,kavis22_high_probab_bound_class_noncon},
but to the best of our knowledge none give a tuning-free algorithm matching SGD
per~\Cref{def:tuning-free-general}. The FindLeader procedure is essentially
extracted from \citep[Theorem 13]{madden20_high_probab_conver_bound_non}, and is
similar to the post-processing step in \citep{ghadimi2013stochastic}.

\textbf{Comparison with the convex setting.} The rate achieved by \Cref{thm:restarted-sgd} stands
in contrast to the best-known rates in the convex setting, where we suffer from
a polynomial dependence on the hints, as in~\cref{eq:50}. One potential reason for this
divergence is the difficulty of telling apart good and bad points. In the convex
setting, we ask for a point $\overline{y}$ with a small function
value $f(\overline{y})$. And while the oracle gives us access to stochastic
realizations of $f(\overline{y})$, the error in those realization is
\emph{not} controlled. Instead, to compare between two points $y_1$ and $y_2$ we
have to rely on stochastic gradient information to approximate
$f(y_1) - f(y_2)$, and this seems to be too difficult without apriori control on
the distance between $y_1$ and $y_2$. On the other hand, in the nonconvex
setting,
such comparison is feasible through sampling methods like e.g. Algorithm~\ref{alg:find-leader}.

\section{Conclusion and Open Problems}

We have reached the end of our investigation. To summarize: we defined
tuning-free algorithms and studied several settings where tuning-free
optimization was possible, and several where we proved impossibility results.
Yet, many open questions remain. For example, tuning-free optimization might be
possible in the finite-sum setting where we can periodically evaluate the
function value exactly. The upper bounds we develop in both the convex and
nonconvex settings require quite stringent assumptions on the noise (such as
boundedness or sub-gaussian norm), and it is not known if they can be relaxed to
expected smoothness~\citep{gower19_sgd,khaled20_better_theor_sgd_noncon_world}
or some variant of it. We leave these questions to future work.

\section*{Acknowledgements}

We thank Aaron Defazio, Yair Carmon, and Oliver Hinder for discussions during the preparation of this work.

\bibliography{reflib}
\bibliographystyle{plainnat}

\clearpage
\onecolumn
\part*{Appendix}

\section{Proofs for \Cref{sec:proof-tuning-free-bounded}}
\label{sec:proof-tuning-free-bounded}

\cutkoskybounded*
\begin{proof}
In the bounded setting, \citet{cutkosky19_artif_const_lipsc_hints_uncon_onlin_learn} give an algorithm that takes as parameters $\epsilon, \alpha$ and achieves the following regret
\begin{align*}
  \sum_{t=0}^{T-1} \ev{g_t, x_t - x_{\ast}} \leq \epsilon + GD &+ \norm{x_{\ast}-x_0} G \log \left[ \frac{\norm{x_{\ast}-x_0} G \exp (\alpha/4G^2)}{\epsilon} \left (1 + \frac{\sum_{t=0}^{T-1} \sqn{g_t}}{\alpha} \right)^{4.5} \right] \\
&+ \norm{x_{\ast}-x_0} \sqrt{\sum_{t=0}^{T-1} \sqn{g_t} \log \left[ \frac{\left( \sum_{t=0}^{T-1} \sqn{g_t} \right)^{10} \exp(\alpha/2G^2) \sqn{x_{\ast}}}{\epsilon^2} + 1 \right]}.
\end{align*}
If we set $\epsilon = \underline{D} \cdot \underline{G}$, $\alpha = \underline{G}^2$, and use
the upper bound $\norm{x_0 - x_{\ast}} \leq D$ and simplify we get the regret
\begin{align*}
  \sum_{t=0}^{T-1} \ev{g_t, x_t - x_{\ast}} &\leq \underline{G} \underline{D} + GD + DG \log \left[  \frac{DG}{\underline{D} \underline{G}} \left( 1 + \frac{G^2 T}{\underline{G}^2}  \right)^{4.5} \right] + D \sqrt{\sum_{t=0}^{T-1} \sqn{g_t}} \sqrt{\log \frac{T^{10} G^{20} D^2}{\underline{D}^2 \underline{G}^2}}
\end{align*}
Observe that because $\underline{G} \leq G$ and $\underline{D} \leq D$ the above can be simplified to
\begin{align*}
  \sum_{t=0}^{T-1} \ev{g_t, x_t - x_{\ast}} &\leq GD \log_+ \left[  \frac{DG}{\underline{D} \underline{G}} \left( 1 + \frac{G^2 T}{\underline{G}^2}  \right)^{4.5} \right] + D \sqrt{\sum_{t=0}^{T-1} \sqn{g_t}} \sqrt{ \log \frac{T^{10} G^{20} D^2}{\underline{D}^2 \underline{G}^2} }
\end{align*}
Call the maximum of the two log terms $\iota$, then the above rate is
\begin{align}
\label{eq:4-c}
  \sum_{t=0}^{T-1} \ev{g_t, x_t - x_{\ast}} &\leq GD \iota + D \sqrt{\sum_{t=0}^{T-1} \sqn{g_t}} \sqrt{\iota}.
\end{align}
Applying online-to-batch conversion starting from~\cref{eq:4-c} proves the
algorithm is tuning-free. For the smooth setting, it suffices to observe that
under a bounded domain we have for any $t$
\begin{align*}
\norm{g_t} &\leq \norm{g_t - \nabla f(x_t)} + \norm{\nabla f(x_t)} \\
&= \norm{g_t - \nabla f(x_t)} + \norm{\nabla f(x_t) - \nabla f(x_{\ast})} \\
&\leq \sigma + L \norm{x_t - x_{\ast}} \\
&\leq \sigma + L D.
\end{align*}
Combining this and following online-to-batch conversion as in
\citep{levy17_onlin_to_offlin_conver_univer} shows the algorithm considered is
tuning-free in the smooth setting as well.

\end{proof}

We will make use of the following two lemmas throughout the upper bound
proofs for DoG and DoWG.

\begin{lemma}\label{lem:dog-concentration}
\citep[Lemma 7]{ivgi23_dog_is_sgds_best_frien}. Let $S$ be the set of nonnegative and nondecreasing sequences. Let $C_t \in \mathcal{F}_{t-1}$ and let $X_t$ be a martingale difference sequence adapted to $\mathcal{F}_t$ such that $\abs{X_t} \leq C_t$ with probability $1$ for all $t$. Then for all $\delta \in (0, 1)$ and $\hat{X}_t \in \mathcal{F}_{t-1}$ such that $\abs{\hat{X}_t} \leq C_t$ with probability $1$, we have that with probability $1-\delta$ that for all $c > 0$
\begin{align*}
  \abs{\sum_{i=1}^t y_i X_i} \leq 8 y_t \sqrt{\theta_{t, \delta} \sum_{i=1}^t (X_i - \hat{X}_i)^2 + \left[ c \right] \theta_{t, \delta}^2} + \pr[ \exists t \leq T \mid C_t > c ]
\end{align*}
\end{lemma}

\begin{lemma}\label{lem:log-term-place}
\citep[Lemma 3]{ivgi23_dog_is_sgds_best_frien}. Let $s_0, s_1, \ldots, s_T$ be a positive increasing sequence. Then,
\begin{align*}
  \max_{t \leq T} \sum_{i=0}^{t-1} \frac{s_i}{s_t}  \geq \frac{1}{e} \left( \frac{T}{\log_+ \frac{s_T}{s_0} } - 1 \right).
\end{align*}
\end{lemma}

\begin{lemma}\label{lem:dog-intermediate-almost-sure}
\citep[Lemma 1]{ivgi23_dog_is_sgds_best_frien}. Suppose that $f$ is convex and
has a minimizer $x_{\ast}$. Then the iterates generated by DoG satisfy for each $t$:
\begin{align*}
   \sum_{k=a}^{b-1} \overline{r}_k \ev{ g_k , x_k - x_{\ast} } \leq \overline{r}_b (2 \overline{d}_b + \overline{r}_b) \sqrt{u_{b-1}}.
\end{align*}
\end{lemma}

\begin{lemma}\label{lem:dowg-intermediate-almost-sure}
Suppose that $f$ is convex and has a minimizer $x_{\ast}$. Then iterates generated by DoWG satisfy for every $t$:
\begin{align*}
\sum_{k=0}^{t-1} \overline{r}_k^2 \ev{\nabla f(x_k), x_k - x_{\ast}} \leq 2 \bar{r}_t \left[ \bar{d}_t + \bar{r}_t \right] \sqrt{v_{t-1}} + \sum_{k=0}^{t-1} \overline{r}_k^2 \ev{\nabla f(x_k) - g_k, x_k - x_{\ast}}
\end{align*}
\end{lemma}
\begin{proof}
This is a modification of \citep[Lemma 3]{khaled23_dowg_unleas} to account for
the case where $g_k \neq \nabla f(x_k)$ (i.e. when the gradients used are not deterministic), following \citep[Lemma 1]{ivgi23_dog_is_sgds_best_frien}. We start
\begin{align*}
d_{k+1}^2 &\leq \sqn{x_k - \eta_k g_k - x_{\ast}} \\
&= \sqn{x_k - x_{\ast}} + \eta_k^2 \sqn{g_k} - 2 \eta_k \ev{g_k, x_k - x_{\ast}}.
\end{align*}
Rearranging we get
\begin{align*}
2 \eta_k \ev{g_k, x_k - x_{\ast}} \leq d_k^2 - d_{k+1}^2 + \eta_k^2 \sqn{g_k}
\end{align*}
Multiplying both sides by $\frac{\overline{r}_k^2}{2 \eta_k}$ we get
\begin{align*}
\overline{r}_k^2 \ev{g_k, x_k - x_{\ast}} \leq \frac{1}{2} \frac{\overline{r}_k^2}{\eta_k}  \left( d_k^2 - d_{k+1}^2 \right) + \frac{\overline{r}_k^2 \eta_k}{2} \sqn{g_k}.
\end{align*}
We then follow the same proof as in \citep[Lemma 3]{khaled23_dowg_unleas} to get
\begin{align}
\label{eq:16}
  \sum_{k=0}^{t-1} \overline{r}_k^2 \ev{g_k, x_k - x_{\ast}} \leq 2 \bar{r}_t \left[ \bar{d}_t + \bar{r}_t \right] \sqrt{v_{t-1}}.
\end{align}
We then decompose
\begin{align*}
\sum_{k=0}^{t-1} \overline{r}_k^2 \ev{g_k, x_k - x_{\ast}} &= \sum_{k=0}^{t-1} \overline{r}_k^2 \ev{g_k - \nabla f(x_k), x_k - x_{\ast}} + \sum_{k=0}^{t-1} \overline{r}_k^2 \ev{\nabla f(x_k), x_k - x_{\ast}}.
\end{align*}
Plugging back into \cref{eq:16} we get
\begin{align}
\label{eq:17}
  \sum_{k=0}^{t-1} \overline{r}_k^2 \ev{\nabla f(x_k), x_k - x_{\ast}} \leq 2 \bar{r}_t \left[ \bar{d}_t + \bar{r}_t \right] \sqrt{v_{t-1}} + \sum_{k=0}^{t-1} \overline{r}_k^2 \ev{\nabla f(x_k) - g_k, x_k - x_{\ast}}
\end{align}
\end{proof}

\subsection{Proof of \Cref{thm:dog-dowg-tuning-free-bounded}}

\dogdowgbounded*
\begin{proof}[Proof of \Cref{thm:dog-dowg-tuning-free-bounded}]
We first handle the case that
$T < 4 \log_+ \frac{\overline{D}}{\underline{D}}$. In this case we just return
$x_0$. If $f$ is $G$-Lipschitz, then by convexity we have
\begin{align*}
f(x_0) - f_{\ast} \leq \ev{ \nabla f(x_0) , x_0 - x_{\ast} } \leq \norm{\nabla f(x_0)} \norm{x_0 - x_{\ast}} \leq GD \leq \frac{2 GD}{\sqrt{T}} \sqrt{\log_+ \frac{\overline{D}}{\underline{D}}}.
\end{align*}
If $f$ is $L$-smooth, then by smoothness we have
\begin{align*}
f(x_0) - f_{\ast} \leq \frac{L}{2} \sqn{x_0 - x_{\ast}} \leq \frac{L D^2}{2} \leq \frac{2 L D^2}{T} \log_+ \frac{\overline{D}}{\underline{D}}.
\end{align*}
Therefore in both cases the point we return achieves a small enough loss almost
surely. Throughout the rest of the proof, we shall assume that $T \geq 4 \log_+ \frac{\overline{D}}{\underline{D}}$.

\textbf{Part 1: DoG}. In the nonsmooth setting, this is a straightforward consequence of \citep[Proposition 3]{ivgi23_dog_is_sgds_best_frien}. In particular, when using DoG with $r_{\epsilon} = \underline{D}$, then Corollary 1 in their work gives that with probability $1-\delta$ there exists some $\tau \in [T]$ and some absolute constant $c > 0$ such that
\begin{align*}
  f(\overline{x}_{\tau}) - f_{\ast} \leq c \cdot \frac{D G}{\sqrt{T}} \log \frac{60 \log 6t}{\delta} \log \frac{2D}{\underline{D}},
\end{align*}
where $\hat{x}_t \eqdef \frac{1}{\sum_{i=0}^{t-1} \overline{r}_i} \sum_{i=0}^{t-1} \overline{r}_i x_i$.

For the smooth setting, we start with \Cref{lem:dog-intermediate-almost-sure} to
get
\begin{align}
\label{eq:1}
  \sum_{k=0}^{t-1} \overline{r}_k \ev{ \nabla f(x_k) , x_k - x_{\ast} } \leq \overline{r}_t \left( 2 \overline{d}_t + \overline{r}_t \right) \sqrt{u_{t-1}} + \sum_{k=0}^{t-1} \overline{r}_k \ev{ \nabla f(x_k) - g_k , x_k - x_{\ast} }.
\end{align}
We follow \citep[Proposition 3]{ivgi23_dog_is_sgds_best_frien} and modify the
proof in a straightforward manner to accommodate the assumption of bounded noise
(rather than bounded gradients). Define
\begin{align*}
\tau_k = \min \left\{ \min \left\{ i \mid \overline{r}_i \geq 2 \overline{r}_{\tau_i-1} \right\}, T \right\}, && \tau_0 \eqdef 0.
\end{align*}
We denote by $K$ the first index such that $\tau_K = T$. Define
\begin{align*}
X_k = \ev{g_k - \nabla f(x_k), \frac{x_k - x_{\ast}}{\bar{d}_k}}, && \hat{X}_k = 0, && y_k = \bar{r}_k \bar{d}_k.
\end{align*}
Observe that $x_k$ is determined by $\mathcal{F}_{k-1}$, and since $\bar{r}_k = \max_{t \leq k} \left( \norm{x_k - x_0}, r_{\epsilon} \right)$, it is also determined by $\mathcal{F}_{k-1}$. Therefore
\begin{align*}
\ec{ X_k \mid \mathcal{F}_{k-1} } = \bar{r}_k^2 \ev{ \ec{ g_k - \nabla f(x_k) } , \frac{x_k - x_{\ast}}{\overline{d}_k}  } = 0.
\end{align*}
Moreover, observe that
\begin{align*}
\abs{X_k} \leq \norm{g_k - \nabla f(x_k)} \frac{\norm{x_k - x_{\ast}}}{\overline{d}_k} \leq \sigma.
\end{align*}
Therefore the $X_k$ form a martingale. Then we can apply Lemma~\ref{lem:dog-concentration} to get that with probability $1-\delta$ that for every $t \in [K]$
\begin{align}
  \abs{\sum_{k=0}^{t-1} \overline{r}_k \ev{ g_k - \nabla f(x_k) , x_k - x_{\ast} }} &\leq 8 \bar{d}_t \bar{r}_t \theta_{t, \delta} \sqrt{\sum_{k=0}^{t-1} (X_k)^2 + \sigma^2} \nonumber\\
&\leq 8 \bar{d}_t \bar{r}_t \theta_{t, \delta} \sqrt{\sum_{k=0}^{t-1} \sqn{g_k - \nabla f(x_k)} + \sigma^2} \nonumber\\
&\leq 8 \bar{d}_t \bar{r}_t \theta_{t, \delta} \sqrt{\sigma^2 t + \sigma^2} \nonumber\\
\label{eq:2}
&\leq 16 \bar{d}_t \bar{r}_t \theta_{t, \delta} \sigma \sqrt{T}.
\end{align}
Now observe that we can use \cref{eq:2} to get
\begin{align}
\abs{\sum_{k=\tau_{i-1}}^{\tau_i-1} \overline{r}_k \ev{g_k - \nabla f(x_k), x_k - x_{\ast}}} &\leq \abs{\sum_{k=0}^{\tau_i-1} \overline{r}_k \ev{ g_k - \nabla f(x_k) , x_k - x_{\ast} }} + \abs{\sum_{k=0}^{\tau_{i-1}-1} \overline{r}_k \ev{ g_k - \nabla f(x_k) , x_k - x_{\ast} } } \nonumber\\
&\leq 16 \bar{d}_{\tau_i} \bar{r}_{\tau_i} \theta_{t, \delta} \sigma \sqrt{T} + 16 \bar{d}_{\tau_{i-1}} \bar{r}_{\tau_{i-1}} \theta_{t, \delta} \sigma \sqrt{T} \nonumber\\
\label{eq:5}
&\leq 32 \overline{d}_{\tau_i} \overline{r}_{\tau_i} \theta_{t, \delta} \sigma \sqrt{T}.
\end{align}
Now observe that by  convexity we have for $k \in \left\{ \tau_{i-1}, \tau_{i-1}+1, \ldots, \tau_i-1 \right\}$
\begin{align*}
  0 \leq f(x_k) - f_{\ast} \leq \ev{\nabla f(x_k), x_k - x_{\ast}} \leq \frac{\overline{r}_k}{\overline{r}_{\tau_{i-1}}} \ev{ \nabla f(x_k) , x_k - x_{\ast} }.
\end{align*}
Summing up from $k=\tau_{i-1}$ to $k=\tau_i-1$ we get
\begin{align}
  \sum_{k=\tau_{i-1}}^{\tau_i-1} \ev{\nabla f(x_k), x_k - x_{\ast}} &\leq \frac{1}{\overline{r}_{\tau_{i-1}}}  \sum_{k=\tau_{i-1}}^{\tau_i-1} \overline{r}_k \ev{\nabla f(x_k), x_k - x_{\ast}} \nonumber\\
&= \frac{1}{\overline{r}_{\tau_{i-1}}}  \sum_{k=\tau_{i-1}}^{\tau_i-1} \overline{r}_k \ev{\nabla f(x_k), x_k - x_{\ast}} \nonumber\\
\label{eq:4}
&= \frac{1}{\overline{r}_{\tau_{i-1}}} \sum_{k=\tau_{i-1}}^{\tau_i-1} \overline{r}_k \ev{\nabla f(x_k) - g_k, x_k - x_{\ast}} + \frac{1}{\overline{r}_{\tau_{i-1}}}  \sum_{k=\tau_{i-1}}^{\tau_i-1} \overline{r}_k \ev{g_k, x_k - x_{\ast}}.
\end{align}
We now use \Cref{lem:dog-intermediate-almost-sure} to get that
\begin{align}
\label{eq:3}
  \sum_{k=\tau_{i-1}}^{\tau_i-1} \overline{r}_k \ev{g_k, x_k - x_{\ast}} \leq 2 \overline{r}_{\tau_i} \left( \overline{d}_{\tau_i} + \overline{r}_{\tau_i} \right) \sqrt{u_{\tau_i-1}}.
\end{align}
Plugging in the upper bounds from \cref{eq:5,eq:3} into \cref{eq:4} we get
\begin{align}
\label{eq:8}
\sum_{k=\tau_{i-1}}^{\tau_i-1} \ev{\nabla f(x_k), x_k - x_{\ast}} \leq \frac{\overline{r}_{\tau_i}}{\overline{r}_{\tau_{i-1}}} \left[ 2 \left( \overline{d}_{\tau_i} + \overline{r}_{\tau_i} \right) \sqrt{u_{\tau_i-1}} + 32 \overline{d}_{\tau_i} \theta_{t, \delta} \sigma \sqrt{T} \right].
\end{align}
Now observe that
\begin{align*}
\overline{r}_{k+1} \leq \overline{r}_k + \norm{x_{t+1} - x_t} &= \overline{r}_k \left( 1 + \frac{\norm{g_k}}{\sqrt{u_k}}  \right) \leq 2 \overline{r}_k.
\end{align*}
It follows that $\frac{\overline{r}_{\tau_i}}{\overline{r}_{\tau_i-1}} \leq 2$. Moreover
by the definition of the $\tau_i$ we have that
$\frac{\overline{r}_{\tau_i-1}}{\overline{r}_{\tau_{i-1}}} \leq 2$. Therefore
\begin{align}
\label{eq:7}
\frac{\overline{r}_{\tau_i}}{\overline{r}_{\tau_{i-1}}} &= \frac{\overline{r}_{\tau_i}}{\overline{r}_{\tau_i-1}} \frac{\overline{r}_{\tau_i-1}}{\overline{r}_{\tau_{i-1}}} \leq 2 \cdot 2 = 4.
\end{align}
using \cref{eq:7} in \cref{eq:8} we get
\begin{align*}
  \sum_{k=\tau_{i-1}}^{\tau_i-1} \ev{\nabla f(x_k), x_k - x_{\ast}} \leq 4 \left[ 2 \left( \overline{d}_{\tau_i} + \overline{r}_{\tau_i} \right) \sqrt{u_{\tau_i-1}} + 32 \overline{d}_{\tau_i} \theta_{t, \delta} \sigma \sqrt{T} \right].
\end{align*}
Summing up over the $i$, we get
\begin{align*}
  \sum_{t=0}^{T-1} \ev{ \nabla f(x_t) , x_t - x_{\ast} } \leq \sum_{i=0}^K \sum_{k=\tau_{i-1}}^{\tau_i-1} \ev{\nabla f(x_k), x_k - x_{\ast}} \leq 4 K \left[ 2 \left( \overline{d}_{\tau_i} + \overline{r}_{\tau_i} \right) \sqrt{u_{\tau_i-1}} + 32 K \overline{d}_{\tau_i} \theta_{t, \delta} \sigma \sqrt{T} \right].
\end{align*}
Observe that by definition we have
\begin{align*}
K \leq 1 + \log \frac{\overline{r}_T}{r_0} &= \log \frac{2 \overline{r}_T}{r_0}.
\end{align*}
Therefore using the last equation and convexity we have
\begin{align*}
  \sum_{t=0}^{T-1} (f(x_t) - f_{\ast}) \leq 4 \log \frac{2 \overline{r}_T}{r_0} \left[ 2 \left( \overline{d}_T + \overline{r}_T \right) \sqrt{u_{T-1}} + 32 \overline{d}_T \theta_{T, \delta} \sigma \sqrt{T} \right].
\end{align*}
Note that because the domain is bounded we have
$\max (\overline{r}_T, \overline{d}_T) \leq D$, and we used $r_0 = \underline{D}$, therefore
\begin{align}
\label{eq:9}
  \sum_{t=0}^{T-1} \left( f(x_t) - f_{\ast} \right) \leq 4 \log \frac{2 D}{\underline{D}} \left[ 4 D \sqrt{u_{T-1}} + 32 D \theta_{T, \delta} \sigma \sqrt{T} \right].
\end{align}
Observe that by our assumption on the noise and smoothness we have
\begin{align*}
  u_{T-1} &= \sum_{k=0}^{T-1} \sqn{g_k} \\
          &\leq 2 \sum_{k=0}^{T-1} \sqn{g_k - \nabla f(x_k)} + 2 \sum_{k=0}^{T-1} \sqn{\nabla f(x_k)} \\
          &\leq 2 T \sigma^2 + 2 \sum_{k=0}^{T-1} \sqn{\nabla f(x_k)} \\
          &\leq 2 T \sigma^2 + 4 L \sum_{k=0}^{T-1} \left( f(x_k) - f_{\ast} \right).
\end{align*}
Using this in \cref{eq:9} gives
\begin{align}
  \sum_{t=0}^{T-1} \left( f(x_t) - f_{\ast} \right) &\leq 4 \log \frac{2D}{\underline{D}} \left[ 8 D \sigma \sqrt{T} + 2 \sqrt{L} D \sqrt{\sum_{t=0}^{T-1} \left( f(x_t) - f_{\ast} \right)} + 32 D \theta_{T, \delta} \sigma \sqrt{T}  \right] \nonumber\\
\label{eq:11-a}
  &\leq 8 \log \frac{2D}{\underline{D}} D \sqrt{L} \sqrt{\sum_{t=0}^{T-1} \left( f(x_t) - f_{\ast} \right)} + 160 \log \frac{2D}{\underline{D}} \theta_{T, \delta} \sigma D \sqrt{T}.
\end{align}
Observe that if $y^2 \leq ay + b$, then by the quadratic equation and the triangle inequality we have
\begin{align*}
y \leq \frac{a + \sqrt{a^2 + 4b}}{2}.
\end{align*}
Squaring both sides gives
\begin{align}
\label{eq:10-c}
y^2 \leq \frac{1}{4} (a + \sqrt{a^2 + 4b})^2 \leq \frac{1}{2} \left( 2a^2 + 4b \right) = a^2 + 2b.
\end{align}
Applying this to \cref{eq:11-a} with the following choices
\begin{align*}
  y = \sqrt{\sum_{t=0}^{T-1} f(x_t) - f_{\ast}}, && a = 8 \log \frac{2D}{\underline{D}} D \sqrt{L}, && b = 160 \log \frac{2D}{\underline{D}}  \theta_{T, \delta} \sigma D \sqrt{T},
\end{align*}
then we obtain
\begin{align*}
  \sum_{t=0}^{T-1} \left( f(x_t) - f_{\ast} \right) \leq 64 \log^2 \frac{2D}{\underline{D}} L D^2 + 320 \log^2 \frac{2D}{\underline{D}} \theta_{T, \delta} \sigma D \sqrt{T}.
\end{align*}
Dividing both sides by $T$ and using Jensen's inequality we finally get
\begin{align*}
  f(\hat{x}_t) - f_{\ast}  &\leq \frac{1}{T} \sum_{t=0}^{T-1} \left( f(x_t) - f_{\ast} \right) \\
&\leq 64 \log^2 \frac{2D}{\underline{D}} \frac{L D^2}{T} + 320 \log^2 \frac{2D}{\underline{D}} \theta_{T, \delta} \frac{\sigma D}{\sqrt{T}}.
\end{align*}
This shows DoG is tuning-free in this setting.

Plugging back into \cref{eq:1} we get with probability $1-\delta$ that
\begin{align*}
\sum_{k=0}^{t-1} \overline{r}_k \ev{ \nabla f(x_k) , x_k - x_{\ast} } \leq \overline{r}_t \left( 2 \overline{d}_t + \overline{r}_t \right) \sqrt{u_{t-1}} + 16 \overline{d}_t \overline{r}_t \theta_{t, \delta} \sigma \sqrt{T}.
\end{align*}
Now we can divide both sides by $\overline{r}_t$ to get
\begin{align*}
  \sum_{k=0}^{t-1} \frac{\overline{r}_k}{\overline{r}_t} \ev{ \nabla f(x_k) , x_k - x_{\ast} } \leq \left( 2 \overline{d}_t + \overline{r}_t \right) \sqrt{u_{t-1}} + 16 \overline{d}_t \theta_{t, \delta} \sigma \sqrt{T}.
\end{align*}

\textbf{Part 2: DoWG}. By \Cref{lem:dowg-intermediate-almost-sure} we have that our iterates satisfy
\begin{align*}
\sum_{k=0}^{t-1} \overline{r}_k^2 \ev{\nabla f(x_k), x_k - x_{\ast}} \leq 2 \bar{r}_t \left[ \bar{d}_t + \bar{r}_t \right] \sqrt{v_{t-1}} + \sum_{k=0}^{t-1} \overline{r}_k^2 \ev{\nabla f(x_k) - g_k, x_k - x_{\ast}}
\end{align*}
Define
\begin{align*}
X_k = \ev{g_k - \nabla f(x_k), \frac{x_k - x_{\ast}}{\bar{d}_k}}, && \hat{X}_k = 0, && y_k = \bar{r}_k^2 \bar{d}_k.
\end{align*}
Observe that $x_k$ is determined by $\mathcal{F}_{k-1}$, and since $\bar{r}_k = \max_{t \leq k} \left( \norm{x_k - x_0}, r_{\epsilon} \right)$, it is also determined by $\mathcal{F}_{k-1}$. Therefore
\begin{align*}
\ec{ X_k \mid \mathcal{F}_{k-1} } = \bar{r}_k^2 \ev{ \ec{ g_k - \nabla f(x_k) } , \frac{x_k - x_{\ast}}{\overline{d}_k}  } = 0.
\end{align*}
Moreover, observe that
\begin{align*}
\abs{X_k} \leq \norm{g_k - \nabla f(x_k)} \frac{\norm{x_k - x_{\ast}}}{\overline{d}_k} \leq \sigma.
\end{align*}
Therefore the $X_k$ form a martingale. Then we can apply Lemma~\ref{lem:dog-concentration} to get that with probability $1-\delta$ that for every $t \in [T]$
\begin{align*}
  \abs{\sum_{k=0}^{t-1} \overline{r}_k^2 \ev{g_k - \nabla f(x_k), x_k - x_{\ast}}} &\leq 8 \bar{d}_t \bar{r}_t^2 \theta_{t, \delta} \sqrt{\sum_{k=0}^{t-1} (X_k)^2 + \sigma^2} \\
&\leq 8 \bar{d}_t \bar{r}_t^2 \theta_{t, \delta} \sqrt{\sum_{k=0}^{t-1} \sqn{g_k - \nabla f(x_k)} + \sigma^2} \\
&\leq 8 \bar{d}_t \overline{r}_t^2 \theta_{t, \delta} \sqrt{\sigma^2 t + \sigma^2} \\
&\leq 16 \bar{d}_t \bar{r}_t^2 \theta_{t, \delta} \sigma \sqrt{T}.
\end{align*}
Plugging this back into \cref{eq:17}
we get
\begin{align}
\label{eq:18}
  \sum_{k=0}^{t-1} \bar{r}_k^2 \ev{\nabla f(x_k), x_k - x_{\ast}} &\leq 2 \bar{r}_t \left[ \bar{d}_t + \bar{r}_t \right] \sqrt{v_{t-1}} + 16 \bar{d}_t \bar{r}_t^2 \theta_{t, \delta} \sigma \sqrt{T}.
\end{align}
We now divide the proof in two cases:
\begin{itemize}
\item If $f$ is $G$-Lipschitz: then
        $\sigma = \sup_{x \in \mathbb{R}^d} \norm{\nabla f(x) - g(x)} \leq 2G$ and therefore
        \cref{eq:18} reduces to
\begin{align*}
  \sum_{k=0}^{t-1} \overline{r}_k^2 \ev{ \nabla f(x_k) , x_k - x_{\ast} } \leq 2 \overline{r}_t \left[ \overline{d}_t + \overline{r}_t \right] \sqrt{v_{t-1}} + 32 \overline{d}_t \overline{r}_t^2 \theta_{t, \delta} G \sqrt{T}.
\end{align*}
And we have
\begin{align*}
  v_{t-1} &= \sum_{k=0}^{t-1} \overline{r}_k^2 \sqn{g_k} \leq \overline{r}_t^2 G^2 T.
\end{align*}
Therefore
\begin{align*}
\sum_{k=0}^{t-1} \overline{r}_k^2 \ev{ \nabla f(x_k) , x_k - x_{\ast} } &\leq 2 \overline{r}_t^2 \left[ \overline{d}_t + \overline{r}_t \right] G \sqrt{T} + 32 \overline{d}_t \overline{r}_t^2 \theta_{t, \delta} G \sqrt{T} \\
&\leq 34 \overline{r}_t^2 \left[ \overline{d}_t + \overline{r}_t  \right] \theta_{t, \delta} G \sqrt{T} \\
&\leq 68 \overline{r}_t^2 D G \sqrt{T} \theta_{t, \delta}.
\end{align*}
Using convexity we have
\begin{align*}
  \sum_{k=0}^{t-1} \overline{r}_k^2 (f(x_k) - f_{\ast}) \leq \sum_{k=0}^{t-1} \overline{r}_k^2 \ev{ \nabla f(x_k) , x_k - x_{\ast} } \leq 68 \overline{r}_t^2 D G \sqrt{T} \theta_{t, \delta}.
\end{align*}
Dividing both sides by $\sum_{k=0}^{t-1} \overline{r}_k^2$ and using Jensen's inequality we get
\begin{align}
  f(\tilde{x}_t) - f_{\ast} &\leq \frac{1}{\sum_{k=0}^{t-1} \overline{r}_k^2} \sum_{k=0}^{t-1} \overline{r}_k^2 (f(x_k) - f_{\ast}) \nonumber\\
\label{eq:12-c}
                         &\le \frac{\overline{r}_t^2}{\sum_{k=0}^{t-1} \overline{r}_k^2} 68 D G \sqrt{T} \theta_{t, \delta}.
\end{align}
We now use \Cref{lem:log-term-place} to conclude that there exists some $t \leq T$
        such that
\begin{align}
\label{eq:13}
  \frac{\overline{r}_t^2}{\sum_{k=0}^{t-1} \overline{r}_k^2} \leq \frac{e}{\left( \frac{T}{2 \log_+ \frac{r_k}{r_{\epsilon}}} - 1 \right)}
\end{align}
Note that by the fact that $\overline{r}_T \leq \overline{D}$,
        $r_0 = \underline{D}$, and that we assume
        $T \geq 4 \log_+ \frac{\overline{D}}{\underline{D}}$ (see the beginning of this
        proof)  we have
\begin{align*}
\frac{T}{2 \log_+ \frac{\overline{r}_T}{\overline{r}_0}} - 1 \geq \frac{T}{2 \log_+ \frac{\overline{D}}{\underline{D}}} - 1 \geq  \frac{T}{4 \log_+ \frac{\overline{D}}{\underline{D}}}.
\end{align*}
Plugging this into \cref{eq:13} we get
\begin{align*}
\frac{\overline{r}_t^2}{\sum_{k=0}^{t-1} \overline{r}_k^2} \leq \frac{4e}{T} \log_+ \frac{\overline{D}}{\underline{D}} \leq \frac{11}{T} \log_+ \frac{\overline{D}}{\underline{D}}.
\end{align*}
Using this in conjunction with \cref{eq:12-c} we thus have that for some $t \leq T$
\begin{align*}
f(\tilde{x}_t) - f_{\ast} \leq \frac{748 DG \theta_{T, \delta}}{\sqrt{T}} \log_+ \frac{\overline{D}}{\underline{D}}.
\end{align*}
\item If $f$ is $L$-smooth: Observe that by straightforward algebra, our
        assumption on the noise, and smoothness
\begin{align*}
  v_{t-1} &= \sum_{k=0}^{t-1} \bar{r}_k^2 \sqn{g_k} \\
&\leq 2 \sum_{k=0}^{t-1} \bar{r}_k^2 \sqn{g_k - \nabla f(x_k)}  + 2 \sum_{k=0}^{t-1} \bar{r}_k^2 \sqn{\nabla f(x_k)} \\
&\leq 2 \bar{r}_t^2 \sigma^2 T + 2 \sum_{k=0}^{t-1} \bar{r}_k^2 \sqn{\nabla f(x_k)} \\
          &\leq 2 \bar{r}_t^2 \sigma^2 T + 4 L \sum_{k=0}^{t-1} \bar{r}_k^2 (f(x_k) - f_{\ast}).
\end{align*}
Using the last line estimate in \cref{eq:18} with the triangle inequality we get
\begin{align*}
  \sum_{k=0}^{t-1} \bar{r}_k^2 \ev{\nabla f(x_k), x_k - x_{\ast}} &\leq 4 \bar{r}_t \left[ \bar{d}_t + \bar{r}_t \right] \left[ \bar{r}_t \sigma \sqrt{T} + \sqrt{L} \sqrt{\sum_{k=0}^{t-1} \bar{r}_k^2 (f(x_k) - f_{\ast})} \right] + 16 \bar{d}_t \bar{r}_t^2 \theta_{t, \delta} \sigma \sqrt{T}.
\end{align*}
By convexity we have
\begin{align*}
\ev{\nabla f(x_k), x_k - x_{\ast}} \geq f(x_k) - f_{\ast}.
\end{align*}
Therefore
\begin{align}
\label{eq:19}
  \sum_{k=0}^{t-1} \bar{r}_k^2 (f(x_k) - f_{\ast}) &\leq 4 \bar{r}_t \left[ \bar{d}_t + \bar{r}_t \right] \sqrt{L} \sqrt{\sum_{k=0}^{t-1} \bar{r}_k^2 (f(x_k) - f_{\ast})} + 20 \bar{r}_t^2  \theta_{t, \delta} \left[ \bar{d}_t + \bar{r}_t \right] \sigma \sqrt{T}.
\end{align}
Observe that if $y^2 \leq ay + b$, then we have shown in \cref{eq:10-c} that
$y^2 \leq a^2 + 2b$. Applying this to \cref{eq:19} with $a=4 \overline{r}_t \left[ \overline{d}_t + \overline{r}_t \right] \sqrt{L}$ and $b=20 \overline{r}_t^2 \theta_{t, \delta} \left[ \overline{d}_t + \overline{r}_t \right] \sigma \sqrt{T}$ gives
\begin{align*}
  \sum_{k=0}^{t-1} \bar{r}_k^2 (f(x_k) - f_{\ast}) &\leq 16 \bar{r}_t^2 \left[ \bar{d}_t + \bar{r}_t \right]^2 L + 40 \bar{r}_t^2 \theta_{t, \delta} \left[ \bar{d}_t + \bar{r}_t \right] \sigma \sqrt{T} \\
&= \bar{r}_t^2 \left( 16 \left[ \bar{d}_t + \bar{r}_t \right]^2 L + 40 \theta_{t, \delta} \left[ \bar{d}_t + \bar{r}_t \right] \sigma \sqrt{T} \right).
\end{align*}
Dividing both sides by $\sum_{k=0}^{t-1} \bar{r}_k^2$ and using Jensen's inequality we get
\begin{align*}
  f(\hat{x}_t) - f_{\ast} \leq \frac{1}{\sum_{k=0}^{t-1} \bar{r}_k^2} \sum_{k=0}^{t-1} \bar{r}_k^2 (f(x_k) - f_{\ast}) &\leq \frac{\bar{r}_t^2}{\sum_{k=0}^{t-1} \overline{r}_k^2}  \left( 16 \left[ \bar{d}_t + \bar{r}_t \right]^2 L + 40 \theta_{t, \delta} \left[ \bar{d}_t + \bar{r}_t \right] \sigma \sqrt{T} \right),
\end{align*}
where $\hat{x}_t = \frac{1}{\sum_{k=0}^{t-1} \overline{r}_k^2}  \sum_{k=0}^{t-1} \overline{r}_k^2 x_k$. We now use \Cref{lem:log-term-place} to conclude that there exists some $\tau \leq T$ such that
\begin{align*}
f(\hat{x}_{\tau}) - f_{\ast} \leq \frac{e}{\left( \frac{T}{2 \log_+ \frac{r_k}{r_{\epsilon}}} - 1 \right)} \left( 16 \left[ \bar{d}_t + \bar{r}_t \right]^2 L + 40 \theta_{\tau, \delta} \left[ \bar{d}_t + \bar{r}_t \right] \sigma \sqrt{T} \right).
\end{align*}
By assumption on $T$ we have $\frac{T}{2 \log_+ \frac{\overline{D}}{\underline{D}}} - 1 \geq \frac{T}{4 \log_+ \frac{\overline{D}}{\underline{D}}}$, therefore
\begin{align*}
f(\hat{x}_{\tau}) - f_{\ast} &\leq \frac{4e \log_+ \frac{\overline{D}}{\underline{D}}}{T} \left( 16 \left[ \bar{d}_t + \bar{r}_t \right]^2 L + 40 \theta_{\tau, \delta} \left[ \bar{d}_t + \bar{r}_t \right] \sigma \sqrt{T} \right) \\
&\leq 700 \theta_{T, \delta} \log_+ \frac{\overline{D}}{\underline{D}} \cdot \left( \frac{L D^2}{T} + \frac{\sigma D}{\sqrt{T}}  \right),
\end{align*}
where in the last line we used that $\max (\overline{d}_t, \overline{r}_t) \leq D$.
\end{itemize}
\end{proof}

\section{Proofs for \Cref{sec:tuning-free-optim-unbounded}}

\subsection{Proof of \Cref{prop:adaptive-polyak-deterministic}}
\adaptivepolyak*
\begin{proof}
By \citep[Theorem 2]{hazan19_revis_polyak_step_size} we have that the point returned by the algorithm $\overline{x}$ satisfies
\begin{align*}
f(\overline{x}) - f_{\ast} \leq \begin{cases}
                           \frac{2GD_{\ast}}{\sqrt{T}} \log_+ \frac{f(x_{\ast}) - \hat{f}_0}{\frac{G D_{\ast}}{\sqrt{T}}} & \text { if  $f$ is $G$-Lipschitz, } \\
                           \frac{2 L D_{\ast}^2}{T} \log_+ \frac{f(x_{\ast}) - \hat{f}_0}{\frac{L D_{\ast}^2}{T}}  & \text { if $f$ is $L$-smooth. }
                          \end{cases}
\end{align*}
provided that $\hat{f}_0 \leq f_{\ast}$, where $\hat{f}_0$ is a parameter supplied to
the algorithm. To get a valid lower bound on $f_{\ast}$, observe that by the
convexity of $f$ we have
\begin{align*}
f(x_0) - f_{\ast} \leq \ev{\nabla f(x_0), x_0 - x_{\ast}} \leq \norm{\nabla f(x_0)} \norm{x_0 - x_{\ast}} \leq \norm{\nabla f(x_0)} \overline{D}.
\end{align*}
It follows that
\begin{align*}
f_{\ast} \geq f(x_0) - \norm{\nabla f(x_0)} \overline{D}.
\end{align*}
And thus we can use $\hat{f}_0 = f(x_0) - \norm{\nabla f(x_0)} \overline{D}$.
\end{proof}

\subsection{Proof of \Cref{prop:dog-dowg-deterministic}}

\dogdowgdetermenistic*
\begin{proof}

This is shown in \citep[Supplementary material section 7]{khaled23_dowg_unleas}
for DoWG. The proof for DoG is similar and we omit it for simplicity.
\end{proof}

\subsection{Proof of \Cref{thm:impossibility-result-smooth}}
\label{sec:proof-impossibility-smooth}
\begin{proof}
Let $\sigma > 0$. Let $L = \sigma T$. Define the functions
\begin{align*}
f_1 (x) &\eqdef \frac{L}{2} x^2 + \sigma x \\
f_2 (x) &\eqdef \frac{L}{2} x^2 - \frac{\sigma}{T-1} x \\
f(x) &\eqdef \frac{1}{T} f_1 (x) + \left( 1 - \frac{1}{T}  \right) f_2 (x) \\
&= \frac{L}{2} x^2.
\end{align*}
 We shall consider the stochastic oracle $\mathcal{O}(f, \sigma_f)$ that returns function values and gradients as follows:
\begin{align*}
\mathcal{O} (f, \sigma_f) (x) \eqdef \left\{ f_z (x), \nabla f_z (x) \right\} &= \begin{cases}
         \left\{ f_1 (x), \nabla f_1 (x) \right\}  & \text { with probability } \frac{1}{T},  \\
         \left\{ f_2 (x), \nabla f_2 (x) \right\}  & \text { with probability } 1 - \frac{1}{T}.
        \end{cases}
\end{align*}
Clearly we have $\ec{ f_z (x) } = f(x)$ and $\ec{ \nabla f_z (x) } = \nabla f(x)$. Moreover,
\begin{align*}
\norm{\nabla f_1 - \nabla f(x)} = \sigma, && \norm{\nabla f_2 (x) - \nabla f(x)} = \frac{\sigma}{T-1} \leq \sigma.
\end{align*}
It follows that $\sigma_f \leq \sigma$. Therefore $\mathcal{O}(f, \sigma_f)$ is a valid stochastic
first-order oracle. This oracle is similar to the one used by
\citet{attia23_sgd_with_adagr_steps} in their lower bound on the convergence of AdaGrad-Norm. The minimizer of the function $f$ is clearly $x_{\ast}^f = 0$.

Let $u \geq 0$, we shall choose it later. Define
\begin{align*}
h_1 (x) &\eqdef  \frac{L}{2} (x-u)^2 + (\sigma - (T-1) Lu)x + \frac{(T-1)L}{2} u^2, \\
h_2 (x) &\eqdef \frac{L}{2} (x-u)^2 + Lux - \frac{\sigma}{T-1} x - \frac{L}{2} u^2, \\
h(x) &\eqdef \frac{L}{2} \left (x - u \right)^2.
\end{align*}
with the oracle $\mathcal{O} (h, \sigma_h)$ given by
\begin{align*}
\mathcal{O} (h, \sigma_h) (x) \eqdef \left\{ h_z (x), \nabla h_z (x) \right\} &= \begin{cases}
         \left\{ h_1 (x), \nabla h_1 (x) \right\}  & \text { with probability } \frac{1}{T},  \\
         \left\{ h_2 (x), \nabla h_2 (x) \right\}  & \text { with probability } 1 - \frac{1}{T}.
        \end{cases}
\end{align*}
Observe that
\begin{align*}
\begin{split}
\ec{ h_z (x) } &= \frac{1}{T} \left[ \frac{L}{2} (x-u)^2 + \sigma x - (T-1) Lux + \frac{(T-1) L}{2} u^2 \right] \\
&\qquad + \frac{T-1}{T} \left[ \frac{L}{2} (x-u)^2 + Lux - \frac{\sigma x}{T-1} - \frac{L}{2} u^2  \right]
\end{split} \\
&= \frac{L}{2} (x-u)^2 + \frac{\sigma x}{T} - \frac{T-1}{T} Lux + \frac{T-1}{T} \frac{L}{2} u^2 + \frac{T-1}{T} Lux - \frac{\sigma x}{T} - \frac{T-1}{T} \frac{L}{2} u^2 \\
&= h(x).
\end{align*}
We can similarly prove that $\ec{ \nabla h_z (x) } = h(x)$. Moreover,
\begin{align*}
\norm{\nabla h_1 (x) - \nabla h(x)} = \norm{\sigma - (T-1) Lu} &\leq \sigma + (T-1) Lu, \\
\norm{\nabla h_2 (x) - \nabla h(x)} = \norm{\frac{-\sigma}{T-1} + Lu} &\leq \frac{\sigma}{T-1}  + Lu.
\end{align*}
It follows that $\sigma_h \leq \sigma + (T-1) Lu$, therefore $\mathcal{O}(h, \sigma_h)$ is a valid stochastic oracle. Finally, observe that the minimizer of $h$ is $x_{\ast}^h = u$.

We fix the initialization $x_0 = v > 0$. Then the initial distance from the optimum for both $f$ and $h$ are:
\begin{align}
\label{eq:lb-dist-estimates}
D_{\ast} (f) = \abs{v - 0} = v, && D_{\ast} (h) = \abs{v-u}.
\end{align}
And recall that
\begin{align}
\label{eq:lb-sigma-estimates}
\sigma_f \leq \sigma, && \sigma_h \leq \sigma + (T-1) Lu.
\end{align}
Observe that both $f$ and $h$ share the same smoothness constant $L$. We supply the algorithm with the following estimates:
\begin{align}
\arraycolsep=0.1\textwidth
\begin{array}{cc}
\underline{L} =  L, & \overline{L} = L, \\
\underline{D} = \min(v, \abs{u-v}), & \overline{D} = \max(v, \abs{u-v}), \\
\underline{\sigma} = \sigma,  & \overline{\sigma} = \sigma + TLu.
\end{array}
\label{eq:lb-hints}
\end{align}
We note that in light of \cref{eq:lb-dist-estimates,eq:lb-sigma-estimates} and
the definitions of $f$ and $h$, the hints given by \cref{eq:lb-hints} are valid for both problems. Now observe the following:
\begin{align*}
h_2 (x) &= \frac{L}{2} (x-u)^2 + Lux - \frac{\sigma}{T-1} x - \frac{L}{2} u^2 \\
&= \frac{L}{2} (x^2 - 2 ux + u^2) + Lux - \frac{\sigma}{T-1} x - \frac{L}{2} u^2 \\
&= \frac{L}{2} x^2 - \frac{\sigma}{T-1} x\\
&= f_2 (x).
\end{align*}
And by the linearity of expectation we have that $\nabla h_2 (x) = \nabla f_2 (x)$. Therefore both oracles $\mathcal{O}(f, \sigma_f)$ and $\mathcal{O}(h, \sigma_h)$ return the same stochastic gradient and stochastic function values with probability $1 - \frac{1}{T}$.

We thus have that over a run of $T$ steps, with probability $(1 - \frac{1}{T})^T \approx e^{-1}$ the algorithm will only get the evaluations $\left\{ h_2 (x), \nabla h_2 (x) \right\}$ from either oracle, and will get the same hints defined in~\cref{eq:lb-hints}. In this setting, the algorithm cannot distinguish whether it is minimizing $h$ or minimizing $f$, and therefore must minimize both. This is the main idea behind this proof: we use that the algorithm is tuning-free, which gives us that the output of the algorithm $x_{\mathrm{out}}$ satisfies with probability $1-\delta$
\begin{align}
\label{eq:10}
h(x_{\mathrm{out}}) - h_{\ast} &\leq c \cdot \mathrm{poly} \left (\log_+ \frac{\overline{L}}{\underline{L}}, \log_+ \frac{\overline{\sigma}}{\underline{\sigma}}, \log_+ \frac{\overline{D}}{\underline{D}}, \log \frac{1}{\delta}, \log T \right) \left( \frac{L D_{\ast}(h)^2}{T} + \frac{\sigma_h D_{\ast} (h)}{\sqrt{T}}   \right).
\end{align}
We shall let $\iota \eqdef \mathrm{poly} \left (\log_+ \frac{\overline{L}}{\underline{L}}, \log_+ \frac{\overline{\sigma}}{\underline{\sigma}}, \log_+ \frac{\overline{D}}{\underline{D}}, \log \frac{1}{\delta}, \log T \right)$ and note that because all of the relevant parameters (the hints, the horizon $T$, and the probability $\delta$) supplied to the algorithm are unchanged for $h$ and $f$, this $\iota$ will be the same for $h$ and $f$. Continuing from \cref{eq:10} and substituting the expressions for $D_{\ast} (h)$ and $\sigma_h$ from \cref{eq:lb-dist-estimates,eq:lb-sigma-estimates} we get
\begin{align*}
h(x_{\mathrm{out}}) - h_{\ast} &\leq c \iota \left( \frac{L (u-v)^2}{T}  + \frac{(\sigma + (T-1) L u) \abs{u-v}}{\sqrt{T}} \right) \\
&\leq c \iota \left( \frac{L (u-v)^2}{T} + \frac{\sigma \abs{u-v}}{\sqrt{T}} + \sqrt{T} L u \abs{u-v}   \right).
\end{align*}
Using the definition of $h$ and the fact that $h_{\ast} = 0$ we have
\begin{align*}
\frac{L}{2} \sqn{x_{\mathrm{out}} - u} &\leq c \iota \left( \frac{L (u-v)^2}{T} + \frac{\sigma \abs{u-v}}{\sqrt{T}} + \sqrt{T} L u \abs{u-v}   \right).
\end{align*}
Multiplying both sides by $\frac{2}{L}$ and then using the definition $L = \sigma T$ we get
\begin{align*}
\sqn{x_{\mathrm{out}} - u} &\leq 2 c \iota \left( \frac{(u-v)^2}{T} + \frac{\sigma \abs{u-v}}{\sqrt{T} L}  + \sqrt{T} u \abs{u-v} \right) \\
&= 2 c \iota \left( \frac{(u-v)^2}{T} + \frac{\abs{u-v}}{T^{\frac{3}{2}}}  + \sqrt{T} u \abs{u-v} \right)
\end{align*}
This gives by taking square roots and using the triangle inequality
\begin{align*}
\abs{x_{\mathrm{out}} - u} \leq \sqrt{2 c \iota} \left( \abs{u-v} T^{-\frac{1}{2}} + \sqrt{\abs{u-v}} T^{-\frac{3}{4}} + T^{\frac{1}{4}} \sqrt{u \abs{u-v}} \right).
\end{align*}
And finally this implies
\begin{align}
\label{eq:12}
x_{\mathrm{out}} \geq u - \sqrt{2 c \iota} \left( \abs{u-v} T^{-\frac{1}{2}} + \sqrt{\abs{u-v}} T^{-\frac{3}{4}} + T^{\frac{1}{4}} \sqrt{u \abs{u-v}} \right).
\end{align}

Similarly, applying the tuning-free guarantees to $f$ and using that $D_{\ast} (f) = v$ we have
\begin{align*}
\frac{L}{2} \sqn{x_{\mathrm{out}}} = f(x_{\mathrm{out}}) - f_{\ast} \leq c \iota \left( \frac{L D_{\ast} (f)^2}{T} + \frac{\sigma D_{\ast} (f)}{\sqrt{T}}  \right) = c \iota \left( \frac{L v^2}{T} + \frac{\sigma v}{\sqrt{T}}  \right)
\end{align*}
This gives
\begin{align*}
\sqn{x_{\mathrm{out}}} \leq 2 c \iota \left( \frac{v^2}{T} + \frac{\sigma v}{\sqrt{T} L}   \right) = 2 c \iota \left( \frac{v^2}{T} + \frac{v}{T^{\frac{3}{2}}} \right).
\end{align*}
Which gives
\begin{align}
\label{eq:11}
x_{\mathrm{out}} \leq \sqrt{2 c \iota} \left( \frac{v}{\sqrt{T}} + \frac{\sqrt{v}}{T^{\frac{3}{4}}}  \right)
\end{align}
Now let us consider the difference between the lower bound on $x_{\mathrm{out}}$ given by \cref{eq:12} and the upper bound given by \cref{eq:11},
\begin{align}
\label{eq:14}
u - \sqrt{2 c \iota} \left( \abs{u-v} T^{-\frac{1}{2}} + \sqrt{\abs{u-v}} T^{-\frac{3}{4}} + T^{\frac{1}{4}} \sqrt{u \abs{u-v}} \right) - \sqrt{2 c \iota} \left( \frac{v}{\sqrt{T}} + \frac{v}{T^{\frac{3}{4}}}  \right)
\end{align}
Let us put $u=v+1$ and $v=T^2$, then \cref{eq:14} becomes
\begin{align}
\label{eq:9-a}
\left( T^2 + 1 \right) - \sqrt{2 c \iota} \left( T^{\frac{-1}{2}} + T^{\frac{-3}{4}} + T^{\frac{1}{4}} \sqrt{T^2+1} \right) - \sqrt{2 c \iota} \left( T^{2-\frac{1}{2}} + T^{2-\frac{3}{4}} \right).
\end{align}
Now observe that
\begin{align*}
\iota &= \mathrm{poly} \left( \log_+ \frac{\overline{L}}{\underline{L}}, \log_+ \frac{\overline{D}}{\underline{D}}, \log_+ \frac{\sigma + TLu}{\sigma}, \log_+ \frac{1}{\delta}, \log T \right) \\
&= \mathrm{poly} \left( \log_+ 1, \log_+ T^2, \log_+ (1+T^2+T^4), \log_+ \frac{1}{\delta}, \log_+ T \right) \\
&= \mathrm{poly} \left( \log_+ T, \log_+ \frac{1}{\delta} \right).
\end{align*}
We set $\delta = \frac{e^{-1}}{4}$, therefore we finally get that $\iota = \mathrm{poly}(\log T)$, plugging back into \cref{eq:9-a} we get that the difference between the lower bound of \cref{eq:12} and the upper bound of \cref{eq:11} is
\begin{align*}
\left( T^2 + 1 \right) - \sqrt{2 c \mathrm{poly}(\log T)} \left( T^{\frac{-1}{2}} + T^{\frac{-3}{4}} + T^{\frac{1}{4}} \sqrt{T^2+1} \right) - \sqrt{2 c \mathrm{poly}(\log T)} \left( T^{2-\frac{1}{2}} + T^{2-\frac{3}{4}} \right).
\end{align*}
It is obvious that for large enough $T$, this expression is positive. Moreover, this situation happens with a positive probability of at least $\frac{e^{-1}}{2}$ since by the union bound
\begin{align*}
\mathrm{Prob}(\text{Algorithm incorrect for $f, h$} \cup \text{Oracle doesn't output all $\left\{ h_2, \nabla h_2 \right\}$}) &\leq 2\delta + \left (1-(1-\frac{1}{T})^T \right) \\
&\lessapprox 1-\frac{e^{-1}}{2}.
\end{align*}
By contradiction, it follows that no algorithm can be tuning-free.
\end{proof}

\subsection{Proof of \Cref{thm:impossibility-result-non-smooth}}
\begin{proof}
We consider the following functions
\begin{align*}
f (x) &= G \abs{x}, \\
f_1 (x) &= G \abs{x} + G x, \\
f_2 (x) &= G \abs{x} - \frac{G}{T-1} x.
\end{align*}
We consider the stochastic oracle $\mathcal{O}(f, \sigma_f)$ that returns function values and gradients as follows:
\begin{align*}
\mathcal{O} (f, \sigma_f) (x) \eqdef \left\{ f_z (x), \nabla f_z (x) \right\} &= \begin{cases}
         \left\{ f_1 (x), \nabla f_1 (x) \right\}  & \text { with probability } \frac{1}{T},  \\
         \left\{ f_2 (x), \nabla f_2 (x) \right\}  & \text { with probability } 1 - \frac{1}{T}.
        \end{cases}
\end{align*}
Clearly we have $\ec{ f_z (x) } = f(x)$ and $\ec{ \nabla f_z (x) } = \nabla f(x)$. It is
also not difficult to prove that $\norm{\nabla f(x) - \nabla f_z (x)} \leq G$. We define a
second function
\begin{align}
\begin{split}
h(x) &= G \abs{x-u}, \\
h_1 (x) &= (2-T) G \abs{x-u} - (T-1) G \abs{x} + G x, \\
h_2 (x) &= G \abs{x} - \frac{G}{T-1} x.
\end{split}\label{eq:42}
\end{align}
And we shall use the oracle $\mathcal{O} (h, \sigma_h)$ given by
\begin{align*}
\mathcal{O} (h, \sigma_h) (x) \eqdef \left\{ h_z (x), \nabla h_z (x) \right\} &= \begin{cases}
         \left\{ h_1 (x), \nabla h_1 (x) \right\}  & \text { with probability } \frac{1}{T},  \\
         \left\{ h_2 (x), \nabla h_2 (x) \right\}  & \text { with probability } 1 - \frac{1}{T}.
        \end{cases}
\end{align*}
By direct computation we have that $\ec{h_z (x)} = h(x)$ and
$\ec{\nabla h_z (x)} = \nabla h(x)$. From the definition of the functions in~\cref{eq:42}
it is immediate that all the gradients and stochastic gradients are bounded by
$GT$. It follows that $\sigma_h \leq G T$. All in all, this shows $\mathcal{O}(h, \sigma_h)$ is a valid
stochastic oracle.

We set $x_0 = 1$, observe that, like in \Cref{thm:impossibility-result-smooth},
with some small but constant probability both oracles return the same gradients
and function values, and therefore the algorithm cannot distinguish between
them. It is therefore forced to approximately minimize both, giving us the guarantee:
\begin{align*}
f(x_{\mathrm{out}}) - f_{\ast} &\leq c \cdot \iota \cdot \frac{G}{\sqrt{T}} \\
h(x_{\mathrm{out}}) - h_{\ast} &\leq c \cdot \iota \cdot \frac{(GT) \abs{1-u}}{\sqrt{T}} = c \iota \abs{1-u} G \sqrt{T}
\end{align*}
This gives
\begin{align}
\label{eq:25}
\abs{x_{\mathrm{out}}} &\leq \frac{c \iota}{\sqrt{T}} \\
\nonumber
\abs{x_{\mathrm{out}} - u} &\leq c \iota \abs{1-u} \sqrt{T}
\end{align}
Let us put $u = 1 - \frac{1}{T}$, then
\begin{align*}
\abs{x - \left( 1 - \frac{1}{T} \right)} &\leq \frac{c \iota}{\sqrt{T}}
\end{align*}
This implies
\begin{align}
\label{eq:26}
  x_{\mathrm{out}} &\geq 1 - \frac{1}{T} - \frac{c \iota}{\sqrt{T}}
\end{align}
And \cref{eq:25} implies
\begin{align}
\label{eq:27}
x_{\mathrm{out}} \leq \frac{c \iota}{\sqrt{T}}
\end{align}
Because $\iota = \mathrm{poly}(\log T)$ (by direct computation), we have that the lower bound on $x_{\mathrm{out}}$ given by \cref{eq:26} exceeds the upper bound on the same iterate given by \cref{eq:27} as $T$ becomes large enough, and we get our contradiction.
\end{proof}

\section{Proofs for \Cref{sec:what-noise-distr}}

We have the two following algorithm-independent lemmas:

\begin{lemma}\label{lem:bernstein-to-hoeffding}
Suppose that $Y$ is a sub-exponential random variable (see \Cref{def:sub-exponential}) with mean $0$ and
sub-exponential modulus $R^2$, i.e.
for all $t > 0$
\begin{align*}
\pr[ \abs{Y} \geq t ] \leq 2 \exp\left (-\frac{t}{R^2} \right).
\end{align*}
Let $Y_1, \ldots, Y_n$ be i.i.d. copies of $Y$. Then with probability $1-\delta$ it holds that
\begin{align*}
  \abs{\frac{1}{n} \sum_{i=1}^n Y_i} \leq c R^2 \left[ \sqrt{\frac{1}{n} \log \frac{2}{\delta}} + \frac{1}{n} \log \frac{2}{\delta} \right],
\end{align*}
where $c > 0$ is an absolute constant.
\end{lemma}
\begin{proof}
By Bernstein's inequality \citep[Corollary 2.8.3]{vershynin18_high_dim_prob} we have
\begin{align*}
  \pr[ \abs{\frac{1}{n} \sum_{i=1}^n Y_i} \geq t ] \leq 2 \exp \left[ - c \min \left( \frac{t^2}{R^4}, \frac{t}{R^2}  \right) n \right],
\end{align*}
for some $c > 0$. Let us set $t$ as follows
\begin{align*}
t &= \begin{cases}
      R^2 \sqrt{\frac{1}{cn} \log \frac{2}{\delta}}  & \text { if } \frac{1}{cn} \log \frac{2}{\delta} < 1,  \\
      R^2 \left[ \frac{1}{cn} \log \frac{2}{\delta} \right] & \text { if } \frac{1}{cn} \log \frac{2}{\delta} \geq 1.
     \end{cases}
\end{align*}
Then
\begin{align*}
\begin{split}
\frac{t}{R^2} &= \begin{cases}
                 \sqrt{\frac{1}{cn} \log \frac{2}{\delta}}  & \text { if } \frac{1}{cn} \log \frac{2}{\delta} < 1, \\
                 \left[ \frac{1}{cn} \log \frac{2}{\delta} \right]  & \text { if }  \frac{1}{cn} \log \frac{2}{\delta} \geq 1.
                 \end{cases}
\end{split}, &&
\frac{t^2}{R^4} = \begin{cases}
\frac{1}{cn} \log \frac{2}{\delta}  & \text { if }  \frac{1}{cn} \log \frac{2}{\delta} < 1, \\
\left[ \frac{1}{cn} \log \frac{2}{\delta}  \right]^2  & \text { if } \frac{1}{cn} \log \frac{2}{\delta} \geq 1.
\end{cases}
\end{align*}
By combining the two cases above we get
\begin{align*}
\min \left( \frac{t}{R^2}, \frac{t^2}{R^4} \right) &= \frac{1}{cn} \log \frac{2}{\delta}.
\end{align*}
Therefore
\begin{align*}
2 \exp \left[ - c \min \left( \frac{t^2}{R^4}, \frac{t}{R^2}  \right) n \right] &= \delta.
\end{align*}
It follows that with probability at least $1-\delta$ we have,
\begin{align*}
\abs{\frac{1}{n} \sum_{i=1}^n Y_i} &\leq \begin{cases}
      R^2 \sqrt{\frac{1}{cn} \log \frac{2}{\delta}}  & \text { if } \frac{1}{cn} \log \frac{2}{\delta} < 1,  \\
      R^2 \left[ \frac{1}{cn} \log \frac{2}{\delta} \right] & \text { if } \frac{1}{cn} \log \frac{2}{\delta} \geq 1.
     \end{cases} \\
&\leq R^2 \left[ \sqrt{\frac{1}{cn} \log \frac{2}{\delta}} + \frac{1}{cn} \log \frac{2}{\delta} \right].  \qedhere
\end{align*}
\end{proof}

Recall the definition of sub-exponential random variables:
\begin{definition}\label{def:sub-exponential}
We call a random variable $Y$ $R$-sub-exponential if
\begin{align*}
\pr[ \abs{Y} \geq t ] \leq 2 \exp \left( \frac{-t}{R} \right)
\end{align*}
for all $t \geq 0$.
\end{definition}

\begin{definition}\label{def:sub-gaussian}
We call a random variable $Y$ $R$-sub-exponential if
\begin{align*}
\pr[ \abs{Y} \geq t ] \leq 2 \exp \left( \frac{-t^2}{R^2} \right)
\end{align*}
for all $t \geq 0$.
\end{definition}

\begin{lemma}\label{lem:sub-exp-sub-gauss}
\citep[Lemma 2.7.7]{vershynin18_high_dim_prob}
A random variable $Y$ is $R$-sub-gaussian if and only if $Y^2$ is $R^2$-sub-exponential.
\end{lemma}

\begin{lemma}\label{lem:centering-sub-exponential}
\citep[Exercise 2.7.10]{vershynin18_high_dim_prob} If $A$ is $E$-sub-exponential then $A-\ec{ A }$ is $c \cdot E$-sub-exponential for
some absolute constant $c$.
\end{lemma}

\begin{lemma}\label{lem:convergence-of-sample-variance}
Suppose that $X$ is a random variable that satisfies the assumptions in
\Cref{def:signal-to-noise} and $X_1, \ldots, X_n$ are all i.i.d. copies of $X$. Then
with probability $1-\delta$ we have that
\begin{align*}
  \abs{\sum_{i=1}^n (\sqn{X_i} - \sigma^2)} \leq c \cdot \sigma^2 \cdot K_{\mathrm{snr}}^{-2} \left[ \sqrt{n \log \frac{1}{\delta}} + \log \frac{1}{\delta} \right].
\end{align*}
\end{lemma}
\begin{proof}
By assumption we have that $\norm{X_i}$ is $R$-sub-gaussian, therefore by
\Cref{lem:sub-exp-sub-gauss} we have that $\norm{X_i}^2$ is
$R^2$-sub-exponential. By \Cref{lem:centering-sub-exponential} we then have that
$\norm{X_i}^2-\sigma^2$ is $c_1 \cdot R^2$-sub-exponential for some absolute constant $c$. By
\Cref{lem:bernstein-to-hoeffding} applied to $Y_i = \norm{X_i}^2 - \sigma^2$ we have
with probability $1-\delta$ that
\begin{align*}
  \abs{\frac{1}{n} \sum_{i=1}^n (\norm{X_i} - \sigma^2)} \leq c_2 \cdot (c_1 R^2) \left[ \sqrt{\frac{1}{n} \log \frac{2}{\delta}} + \frac{1}{n} \log \frac{2}{\delta}  \right],
\end{align*}
where $c_2 > 0$ is some absolute constant. Using the definition of the signal-to-noise ratio
$K_{\mathrm{snr}}^{-1} = \frac{R}{\sigma}$ we get that for some absolute constant $c$
\begin{align*}
\abs{\frac{1}{n} \sum_{i=1}^n (\norm{X_i} - \sigma^2)} \leq c \cdot \sigma^2 \cdot K_{\mathrm{snr}}^{-2} \left[ \sqrt{\frac{1}{n} \log \frac{2}{\delta}} + \frac{1}{n} \log \frac{2}{\delta}  \right].
\end{align*}
\end{proof}

\subsection{Proof of Theorem~\ref{thm:dog-dowg-special-noise}}

The main idea in the proof is the following lemma, which characterizes the
convergence of the sample variance estimator of $b$ i.i.d. random variables by
the number of samples $b$ as well as the signal-to-noise ratio $K_{\mathrm{snr}}^{-1}$.
\begin{lemma}\label{lem:sample-variance-estimation}
Let $Y$ be a random vector in $\mathbb{R}^d$ such that $Z=Y-\ec{Y}$ satisfies the
assumptions in \Cref{def:signal-to-noise}. Let $Y_1, Y_2, \ldots, Y_b$ be i.i.d.
copies of $Y$. Define the sample mean and variance as
\begin{align*}
  \hat{Y} = \frac{1}{b} \sum_{i=1}^b Y_i, && \hat{\sigma}^2 &= \frac{1}{b} \sum_{i=1}^b \sqn{Y_i - \bar{Y}}.
\end{align*}
Then it holds with probability $1-\delta$ that
\begin{align*}
\abs{\frac{\hat{\sigma}^2}{\sigma^2}  - 1} \leq c \cdot K_{\mathrm{snr}}^{-2} \cdot \left( \sqrt{\frac{\log \frac{2b}{\delta}}{b}} + \frac{\log \frac{2 (b \lor d)}{\delta}}{b}  \right),
\end{align*}
where $c$ is an absolute (non-problem-dependent) constant,
$b \lor d= \eqdef \max(b, d)$,
$\sigma^2 \eqdef \ecn{Y-\ec{Y}}$, and $K_{\mathrm{snr}}^{-2}$ is the ratio defined in \Cref{def:signal-to-noise}.
\end{lemma}
\begin{proof}
We shall use the shorthand $\mu = \ec{ Y }$. We have
\begin{align*}
  \hat{\sigma}^2 &= \frac{1}{b} \sum_{i=1}^b \sqn{Y_i - \hat{Y}} \\
&= \frac{1}{b} \sum_{i=1}^b \sqn{Y_i - \mu + \mu - \hat{Y}} \\
&= \frac{1}{b} \sum_{i=1}^b \left[ \sqn{Y_i - \mu} + \sqn{\mu - \hat{Y}} + 2 \ev{Y_i - \mu, \mu - \hat{Y}} \right] \\
              &= \frac{1}{b} \sum_{i=1}^b \sqn{Y_i - \mu} + \sqn{\mu - \hat{Y}} - \frac{2}{b} \sum_{i=1}^b \ev{Y_i - \mu, \hat{Y} - \mu}
\end{align*}
We have by the triangle inequality
\begin{align}
\label{eq:20}
  \abs{\hat{\sigma}^2 - \sigma^2} &\leq \abs{ \frac{1}{b} \sum_{i=1}^b \sqn{Y_i - \mu} - \sigma^2} + \sqn{\mu - \hat{Y}} + \abs{\frac{2}{b} \sum_{i=1}^b \ev{ Y_i - \mu , \hat{Y}^i - \mu}}
\end{align}
By \Cref{lem:convergence-of-sample-variance}, we may bound the first term on the
right hand side of \cref{eq:20} as
\begin{align}
\label{eq:6}
\abs{ \frac{1}{b} \sum_{i=1}^b \sqn{Y_i - \mu} - \sigma^2} \leq c \cdot \sigma^2 \cdot K_{\mathrm{snr}}^{-2} \left[ \sqrt{\frac{\log \frac{1}{\delta}}{b}} + \frac{\log \frac{2}{\delta}}{b} \right].
\end{align}
For the second term on the right hand side of \cref{eq:20}, we apply
\citep[Corollary 7]{jin19_short_note_concen_inequal_random} to $X_i = \mu - Y_i$
and obtain
\begin{align*}
  \norm{\sum_{i=1}^b \left[ \mu - Y_i \right]} \leq c \cdot \sqrt{b R^2 \log \frac{2d}{\delta}} = c R \sqrt{b \log \frac{2d}{\delta}}.
\end{align*}
Squaring both sides we get
\begin{align*}
  \sqn{\sum_{i=1}^b \left[ \mu - Y_i \right]} \leq c^2 R^2 b \log \frac{2d}{\delta}
\end{align*}
Therefore
\begin{align}
\label{eq:15}
  \sqn{\frac{1}{b} \sum_{i=1}^b \left[ \mu - Y_i \right]} \leq \frac{c^2 R^2 \log \frac{2d}{\delta}}{b}.
\end{align}
For the third term on the right hand side of \cref{eq:20} we have
\begin{align*}
\sum_{i=1}^b \ev{ Y_i - \mu , \hat{Y} - \mu } &= \sum_{i=1}^b \ev{ Y_i - \mu, \frac{1}{b} \sum_{j=1}^b [Y_i - \mu] } \\
&= \frac{1}{b} \sqn{Y_i - \mu} + \frac{1}{b} \sum_{j \neq i} \ev{ Y_i - \mu, Y_j - \mu }.
\end{align*}
Taking absolute values of both sides and using the triangle inequality we get
\begin{align}
  \abs{\frac{1}{b} \sum_{i=1}^b \ev{ Y_i - \mu , \hat{Y} - \mu }} &= \abs{\frac{1}{b} \sqn{Y_i - \mu} + \frac{1}{b} \sum_{j \neq i} \ev{ Y_i - \mu, Y_j - \mu }} \nonumber\\
\label{eq:21}
&\leq \frac{1}{b} \sqn{Y_i - \mu} + \abs{\frac{1}{b} \sum_{j \neq i} \ev{ Y_i - \mu, Y_j - \mu }}.
\end{align}
By our sub-gaussian assumption on $\norm{Y-\mu}$, the first term on the right hand
side of \cref{eq:21} can be bounded with high probability as
\begin{align}
\label{eq:22}
\norm{Y_i - \mu} \leq c \sqrt{R^2 \log \frac{2}{\delta}} = c R \sqrt{\log \frac{2}{\delta}}.
\end{align}
Define $Z_{i, j} = \ev{ Y_i - \mu,  Y_j - \mu}$. Observe that for each $i$, we have that the random vectors
$Z_{i, 1}, \ldots, Z_{i,i-1}, Z_{i,i+1}, \cdots, Z_{i, n}$ are all independent, and
therefore $\ec{Y_{i, j}} = 0$ for $i \neq j$. Observe that by the Cauchy-Schwartz inequality
\begin{align*}
\abs{Z_{i, j}} &= \abs{\ev{ Y_i - \mu,  Y_j - \mu}} \leq \norm{Y_i - \mu} \norm{Y_j - \mu}.
\end{align*}
Observe that each of $\norm{Y_i - \mu}$ and $\norm{Y_j - \mu}$ is
sub-gaussian with modulus $R$, therefore by \citep[Lemma
2.7.7]{vershynin18_high_dim_prob} their product is sub-exponential with modulus
$R^2$. It follows that
\begin{align*}
\pr[ \abs{Z_{i, j}} \geq t ] \leq \pr[ \norm{Y_i - \mu} \norm{Y_j - \mu} \geq t] \leq 2 \exp \left( -\frac{t}{R^2}  \right).
\end{align*}
Therefore $Z_{i, j}$ is also sub-exponential with modulus $R^2$. By
\Cref{lem:bernstein-to-hoeffding} we then get that for any fixed $i$, with
probability at least $1-\delta$ we have
\begin{align}\label{eq:23}
  \abs{\frac{1}{b-1} \sum_{\substack{j=1,\ldots, b \\ j\neq i}} Z_{i, j}} \leq c \cdot R^2 \left[ \sqrt{\frac{1}{b-1} \log \frac{2}{\delta}} + \frac{1}{b-1} \log \frac{2}{\delta} \right],
\end{align}
for some absolute constant $c > 0$. Multiplying both sides of \cref{eq:23} by $\frac{b-1}{b}$ and then using
straightforward algebra we get
\begin{align}
\abs{\frac{1}{b} \sum_{\substack{j=1,\ldots, b \nonumber\\ j\neq i}} Z_{i, j}} &\leq c R^2 \left[ \sqrt{\frac{1}{b-1} \log \frac{2}{\delta}} + \frac{1}{b-1} \log \frac{2}{\delta} \right] \cdot \frac{b-1}{b} \nonumber\\
&= c R^2 \left[ \sqrt{\frac{1}{b} \log \frac{2}{\delta}} \cdot \sqrt{\frac{b}{b-1}} + \frac{1}{b} \log \frac{2}{\delta} \cdot \frac{b}{b-1}  \right] \frac{b-1}{b} \nonumber\\
&= c R^2 \left[ \sqrt{\frac{1}{b} \log \frac{2}{\delta}} \sqrt{\frac{b-1}{b}} + \frac{1}{b} \log \frac{2}{\delta} \right] \nonumber\\
\label{eq:28}
&\leq c R^2 \left[ \sqrt{\frac{1}{b} \log \frac{2}{\delta}} + \frac{1}{b} \log \frac{2}{\delta} \right].
\end{align}
We now use the union bound over all $i$ with the triangle inequality to get
\begin{align}\label{eq:29}
  \abs{\frac{1}{b} \sum_{i=1}^n \frac{1}{b} \sum_{\substack{j=1,\ldots, b\\ j\neq i}} Z_{i, j}} \leq \frac{1}{b} \sum_{i=1}^n \abs{\frac{1}{b} \sum_{\substack{j=1,\ldots, b \\ j\neq i}} Z_{i, j} } \leq c R^2 \left[ \sqrt{\frac{1}{b} \log \frac{2b}{\delta}} + \frac{1}{b} \log \frac{2b}{\delta} \right].
\end{align}
Combining \cref{eq:22,eq:29} we get that with probability $1-\delta$ there exists
some absolute constant $c^{\prime} > 0$
\begin{align}
\label{eq:30}
  \abs{\frac{1}{b} \sum_{i=1}^b \ev{ Y_i - \mu , \hat{Y} - \mu }} \leq c^{\prime} R^2 \left[ \sqrt{\frac{\log \frac{2b}{\delta}}{b}} + \frac{\log \frac{2b}{\delta}}{b}  \right].
\end{align}
Combining \cref{eq:6,eq:15,eq:30} into \cref{eq:20} we get
\begin{align*}
\abs{\hat{\sigma}^2 - \sigma^2} &\leq \abs{ \frac{1}{b} \sum_{i=1}^b \sqn{Y_i - \mu} - \sigma^2} + \sqn{\mu - \hat{Y}} + \abs{\frac{2}{b} \sum_{i=1}^b \ev{ Y_i - \mu , \hat{Y}^i - \mu}} \\
&\leq c_1 \cdot \sigma^2 \cdot K_{\mathrm{snr}}^{-2} \left[ \sqrt{\frac{\log \frac{2}{\delta}}{b}} + \frac{\log \frac{2}{\delta}}{b} \right] + c_2 \frac{R^2 \log \frac{2d}{\delta}}{b} + c_3 R^2 \left[ \sqrt{\frac{\log \frac{2b}{\delta}}{b}} + \frac{\log \frac{2b}{\delta}}{b}  \right].
\end{align*}
For some absolute constants $c_1, c_2, c_3 > 0$. Therefore, using the definition
$K_{\mathrm{snr}}^{-1} = \frac{R}{\sigma}$ and simplifying in the last equation we finally
get that
\begin{align*}
\abs{\hat{\sigma}^2 - \sigma^2} &\leq c_4 \cdot \sigma^2 K_{\mathrm{snr}}^{-2} \left[ \sqrt{\frac{\log \frac{2b}{\delta}}{b}} + \frac{\log \frac{2 (b \lor d)}{\delta}}{b}   \right],
\end{align*}
for some absolute constant $c_4 > 0$. Dividing both sides by $\sigma^2$ yields the
statement of the lemma.
\end{proof}

\begin{algorithm*}
\caption{T-DoG + Variance Estimation}
\label{alg:dog-with-estimation}
\begin{algorithmic}[1]
\REQUIRE initial point $x_0 \in \mathcal{X}$, initial distance estimate $r_{\epsilon} > 0$,
minibatch size $b$, $\theta > 0$.
\STATE Initialize $r_{\epsilon} = \underline{D}$, $\alpha = 8^4 \cdot \log (60 \log (6T)/\delta) \cdot \theta^{-1}$.
\FOR{$t = 0, 1, 2, \dots, T-1$}
\STATE Update distance estimator:
$\bar{r}_t \gets \max \left( \norm{x_t - x_0}, \ \bar{r}_{t-1} \right)$.
\STATE Sample $b$ stochastic gradients $\mu_t^1, \mu_t^2, \ldots, \mu_t^b$ at  $x_t$ and compute:
\begin{align*}
  \hat{\mu}_t &= \frac{1}{b} \sum_{i=1}^b \mu_t^i, && \hat{\sigma}_t^2 = \frac{1}{b} \sum_{i=1}^b \sqn{\mu_t^i - \hat{\mu}_t^i}, && \overline{\sigma}_t^2 = \max_{k \le t} \hat{\sigma}_k^2.
\end{align*}
\STATE Compute a new stochastic gradient $g_t$ evaluated at $x_t$.
\STATE Update the gradient sum $u_t = u_{t-1} + \sqn{g_t}$.
\STATE Set the stepsize:
\begin{align}
\label{eq:t-dog-update}
\eta_t \gets \frac{\bar{r}_t}{\alpha \sqrt{u_t + \beta \bar{\sigma}_t^2}} \frac{1}{\log_+^2 \left( 1 + \frac{u_t + \bar{\sigma}_t^2}{v_0 + \bar{\sigma}_0^2} \right)}.
\end{align}
\STATE Gradient descent step: $x_{t+1} \gets x_t - \eta_t \nabla f(x_t)$.
\ENDFOR
\end{algorithmic}
\end{algorithm*}

\begin{algorithm}
\caption{T-DoWG + Variance Estimation}
\label{alg:dowg-with-estimation}
\begin{algorithmic}[1]
\REQUIRE initial point $x_0 \in \mathcal{X}$, initial distance estimate $r_{\epsilon} > 0$,
minibatch size $b$, $\theta > 0$.
\STATE Initialize $r_{\epsilon} = \underline{D}$, $\alpha = 8^4 \cdot \log (60 \log (6T)/\delta) \cdot \theta^{-1}$.
\FOR{$t = 0, 1, 2, \dots, T-1$}
\STATE Update distance estimator:
$\bar{r}_t \gets \max \left( \norm{x_t - x_0}, \ \bar{r}_{t-1} \right)$.
\STATE Sample $b$ stochastic gradients $\mu_t^1, \mu_t^2, \ldots, \mu_t^b$ at  $x_t$ and compute:
\begin{align*}
  \hat{\mu}_t &= \frac{1}{b} \sum_{i=1}^b \mu_t^i, && \hat{\sigma}_t^2 = \frac{1}{b} \sum_{i=1}^b \sqn{\mu_t^i - \hat{\mu}_t^i}, && \overline{\sigma}_t^2 = \max_{k \le t} \hat{\sigma}_k^2.
\end{align*}
\STATE Compute a new stochastic gradient $g_t$ evaluated at $x_t$.
\STATE Update weighted gradient sum: $v_t \gets v_{t-1} + \bar{r}_t^2 \norm{g_t}^2$.
\STATE Set the stepsize:
\begin{align}
\label{eq:t-dowg-update}
\gamma_t \gets \frac{\bar{r}_t^2}{\alpha \sqrt{v_t + \beta \overline{r}_t^2 \bar{\sigma}_t^2}} \frac{1}{\log_+^2 \left( 1 + \frac{v_t + \overline{r}_t^2 \bar{\sigma}_t^2}{v_0 + \overline{r}_0^2 \bar{\sigma}_0^2} \right)}.
\end{align}
\STATE Gradient descent step: $x_{t+1} \gets x_t - \gamma_t \nabla f(x_t)$.
\ENDFOR
\end{algorithmic}
\end{algorithm}

\begin{proof}[Proof of Theorem~\ref{thm:dog-dowg-special-noise}]
First, observe that at every timestep $t$, conditioned on
$\mathcal{F}_t = \sigma \left( g_{1:t-1}, x_{1:t} \right)$ we have by
\Cref{lem:sample-variance-estimation} that with probability $1-\frac{\delta}{T}$ that the sample variance $\hat{\sigma}^2_t$
satisfies for some $c > 0$
\begin{align*}
\abs{\frac{\hat{\sigma}_t^2}{\sigma^2 (x_t)} - 1} \leq c \cdot K_{\mathrm{snr}}^{-2} \cdot \left( \sqrt{\frac{\log \frac{2bT}{\delta}}{b}} + \frac{\log \frac{ 2(b \lor d)T}{\delta}}{b}  \right),
\end{align*}
where $c$ is an absolute constant and $\sigma_t^2 = \sigma^2 (x_t)$ denotes the variance of the
noise at $x_t$ (we do not assume that the noise distribution is the same for all
$t$). By our assumption on the minibatch size we have that for some $u \in [0, K_{\mathrm{snr}}^2]$
\begin{align*}
c \cdot K_{\mathrm{snr}}^{-2} \cdot \left( \sqrt{\frac{\log \frac{2bT}{\delta}}{b}} + \frac{\log \frac{ 2(b \lor d) T}{\delta}}{b}  \right) \leq 1 - \frac{\theta}{K_{\mathrm{snr}}^2}.
\end{align*}
And therefore
\begin{align*}
\abs{\frac{\hat{\sigma}_t^2}{\sigma^2 (x_t)} - 1} \leq 1 - \frac{\theta}{K_{\mathrm{snr}}^2}.
\end{align*}
Which gives
\begin{align*}
\frac{\hat{\sigma}_t^2}{\sigma_t^2} \ge 1 - \left( 1 - \frac{\theta}{K_{\mathrm{snr}}^2} \right) = \frac{\theta}{K_{\mathrm{snr}}^2}.
\end{align*}
Multiplying both sides by $\sigma_t^2$ we get
\begin{align*}
\hat{\sigma}_t^2 &\geq \sigma_t^2 \frac{\theta}{K_{\mathrm{snr}}^2} \geq R^2 \theta.
\end{align*}
Therefore $\hat{\sigma}_t^2/\theta$ is, with high probability, an upper bound on any noise
norm, and we can use that as normalization in T-DoG/T-DoWG. This is the key idea
of the proof, and it's entirely owed to \Cref{lem:sample-variance-estimation}.
The rest of the proof follows \citep{ivgi23_dog_is_sgds_best_frien} with only a
few changes to incorporate the variance estimation process.

Following \citep{ivgi23_dog_is_sgds_best_frien}, we define the stopping time
\begin{align*}
\mathcal{T}_{\mathrm{out}} &= \min \left\{ t \mid \overline{r}_t > 3 d_0 \right\}.
\end{align*}
And define the proxy sequences
\begin{align}
\label{eq:39}
\tilde{\eta}_k &= \begin{cases}
               \eta_k  & \text { if } k < \mathcal{T}_{\mathrm{out}}, \\
               0  & \text { otherwise. }
\end{cases} && \tilde{\gamma}_k = \begin{cases}
                              \gamma_k  & \text { if } k < \mathcal{T}_{\mathrm{out}}, \\
                              0  & \text { otherwise. }
                              \end{cases}
\end{align}

\begin{lemma}(Modification of \citep[Lemma 8]{ivgi23_dog_is_sgds_best_frien})
Under the conditions of Theorem~\ref{thm:dog-dowg-special-noise} both the
DoG~\eqref{eq:t-dog-update} and DoWG~\eqref{eq:t-dowg-update} updates satisfy
for all $t \leq T$
\begin{align*}
\rho_t &\in \sigma(g_0,  \mu_0^1, \ldots, \mu_0^b \ldots, g_{t-1}, \mu_0^{t-1}, \ldots, \mu_b^{t-1}), \\
\abs{\rho_t \ev{ g_t - \nabla f(x_t), x_t - x_{\ast} }} &\leq \frac{6 d_0^2}{8^2 \theta_{T, \delta}}, \\
\sum_{k=0}^t \rho_k^2 \sqn{g_k} &\leq \frac{9 d_0^2}{8^4 \theta_{T, \delta}}, \\
  \sum_{k=0}^t (\rho_k \ev{ g_k ,  x_k - x_{\ast}})^2 &\leq \frac{12^2 d_0^4}{8^4 \theta_{T, \delta}},
\end{align*}
where $\rho_t$ stands for either the DoG stepsize proxy $\tilde{\eta}_k$ or the DoWG
stepsize proxy $\tilde{\gamma}_k$.
\end{lemma}
\begin{proof}
The modification of this lemma to account for bounded noise $g(x_k) - \nabla f(x_k)$ rather than
bounded gradients is straightforward, and we omit it for simplicity.
\end{proof}

\begin{lemma}\label{lem:dog-dowg-stoch-step-2}(Modification of \citep[Lemma 9]{ivgi23_dog_is_sgds_best_frien})
Under the conditions of Theorem~\ref{thm:dog-dowg-special-noise} both the
DoG~\eqref{eq:t-dog-update} and DoWG~\eqref{eq:t-dowg-update} updates satisfy
for all $t \leq T$ with probability at least $1-\delta$
\begin{align*}
  \sum_{k=0}^{t-1} \tilde{\eta}_k \ev{ g_k - \nabla f(x_k) , x_{\ast} - x_k } \leq d_0^2.
\end{align*}
\end{lemma}
\begin{proof}
The modification is straightforward and omitted.
\end{proof}

\begin{lemma}(Modification of \citep[Lemma 10]{ivgi23_dog_is_sgds_best_frien})\label{lem:dog-dowg-stoch-final-step}
  Under the conditions of Theorem~\ref{thm:dog-dowg-special-noise}, if
  $\sum_{k=0}^{t-1} \rho_t \ev{ g_k - \nabla f(x_k) , x_{\ast} - x_k } \leq d_0^2$ for all
  $t \leq T$, then $\mathcal{T}_{\mathrm{out}} > T$.
\end{lemma}
\begin{proof}
The modification is straightforward and omitted.
\end{proof}
By Lemmas~\cref{lem:dog-dowg-stoch-step-2,lem:dog-dowg-stoch-final-step} we get
that $\overline{r}_T \leq 3 d_0$ and it follows that
$\overline{d}_t = \max_{k \leq t} d_k \leq \max_{k \leq t} r_t + r_0 \leq 4 d_0$. Then, a
straightforward modification of~\Cref{thm:dog-dowg-tuning-free-bounded} to
handle the slightly smaller stepsizes used by T-DoG/T-DoWG shows that both
methods are tuning-free. The proof is very similar
to~\Cref{thm:dog-dowg-tuning-free-bounded} and is omitted.
\end{proof}

\section{Proofs for \Cref{sec:nonconvex-tuning-free}}

\subsection{Proof of \Cref{thm:impossibility-result-nonconvex}}
\begin{proof}
We use the exact same construction from \Cref{thm:impossibility-result-smooth} with
the following hints:
\begin{align*}
\underline{L} = L, && \overline{L} = L \\
\underline{\Delta} = \frac{L}{2} \min(v, \abs{u-v}), && \overline{\Delta} = \frac{L}{2} \max(v, \abs{u-v}). \\
\underline{\sigma} = \sigma, && \overline{\sigma} = \sigma + T L u,
\end{align*}
where $u > 0$ and $v > 0$ are parameters we shall choose later. Suppose that we
have that the algorithm's output point $x$ satisfies
\begin{align*}
\sqn{\nabla f(x)} \leq c \iota \left[ \sqrt{\frac{L (f(x_0) - f_{\ast}) \sigma^2}{T}} + \frac{L (f(x_0) - f_{\ast})}{T} \right].
\end{align*}
We now use the fact that $f(x_0) - f_{\ast} = \frac{L}{2} (x- x_{\ast})^2$ to get
\begin{align*}
L^2 \sqn{x_{\mathrm{out}}-x_{\ast}} &= \sqn{\nabla f(x)} \\
&\leq c \iota \sqrt{\frac{L^2 (x_0 - x_{\ast})^2 \sigma_f^2}{T}} + c\iota \frac{L^2 (x_0 - x_{\ast})^2}{T}  \\
&= c \iota \frac{L \abs{x_0-x_{\ast}} \sigma_f}{\sqrt{T}} + c \iota \frac{L^2 (x_0 - x_{\ast})^2}{T}.
\end{align*}
Dividing both sides by $L^2$ we get
\begin{align*}
\sqn{x_{\mathrm{out}}-x_{\ast}} \leq c \iota \frac{\abs{x_0 - x_{\ast}} \sigma_f}{L \sqrt{T}} + c \iota \frac{\sqn{x_0 - x_{\ast}}}{T}.
\end{align*}
Taking square roots and using the triangle inequality gives
\begin{align}
\label{eq:44}
\abs{x_{\mathrm{out}} - x_{\ast}} \leq \sqrt{c \iota} \sqrt{\abs{x_0 - x_{\ast}}} \sqrt{\frac{\sigma_f}{L}} \frac{1}{T^{\frac{1}{4}}} + \sqrt{c \iota} \frac{\abs{x_0 - x_{\ast}}}{\sqrt{T}}.
\end{align}
Applying \cref{eq:44} to the function $f$ with $x_0 = v > 0$, $x_{\ast} = 0$, $\sigma_f = \sigma$, and $L = \sigma \sqrt{T}$ we get
\begin{align*}
\abs{x_{\mathrm{out}}} \leq \sqrt{c \iota} \sqrt{\frac{v}{T}} + \sqrt{c \iota} \frac{v}{\sqrt{T}}.
\end{align*}
Therefore
\begin{align}
\label{eq:46}
x_{\mathrm{out}} \leq \sqrt{c \iota} \sqrt{\frac{v}{T}} + \sqrt{c \iota} \frac{v}{\sqrt{T}}.
\end{align}
On the other hand, applying \cref{eq:44} to the function $h = \frac{L}{2} (x-u)^2$
(as in the proof of \Cref{thm:impossibility-result-smooth}) we obtain
\begin{align*}
\abs{x_{\mathrm{out}} - u} &\leq \sqrt{c \iota} \sqrt{\abs{u-v}} \sqrt{\frac{(\sigma + L T u)}{L}} \frac{1}{T^{\frac{1}{4}}} + \sqrt{c \iota} \frac{\abs{u-v}}{\sqrt{T}}  \\
&= \sqrt{c \iota} \sqrt{\abs{u-v}} \sqrt{\frac{1}{\sqrt{T}} + T u } \frac{1}{T^{\frac{1}{4}}} + \sqrt{c \iota} \frac{\abs{u-v}}{\sqrt{T}} \\
&\leq \sqrt{c \iota}\sqrt{\abs{u-v}} \left[ \frac{1}{T^{\frac{1}{4}}} + \sqrt{T u}  \right] \frac{1}{T^{\frac{1}{4}}}+ \sqrt{c \iota} \frac{\abs{u-v}}{\sqrt{T}}  \\
&= \sqrt{c \iota} \frac{\sqrt{\abs{u-v}}}{\sqrt{T}} + \sqrt{c \iota} T^{\frac{1}{4}} \sqrt{u \abs{u-v}} + \sqrt{c \iota} \frac{\abs{u-v}}{\sqrt{T}}.
\end{align*}
Therefore
\begin{align}
\label{eq:47}
x_{\mathrm{out}} \geq u - \left[ \sqrt{c \iota} \frac{\sqrt{\abs{u-v}}}{\sqrt{T}} + \sqrt{c \iota} T^{\frac{1}{4}} \sqrt{u \abs{u-v}} + \sqrt{c \iota} \frac{\abs{u-v}}{\sqrt{T}} \right].
\end{align}
Combining \cref{eq:46,eq:47} gives
\begin{align*}
u - \left[ \sqrt{c \iota} \frac{\sqrt{\abs{u-v}}}{\sqrt{T}} + \sqrt{c \iota} T^{\frac{1}{4}} \sqrt{u \abs{u-v}} + \sqrt{c \iota} \frac{\abs{u-v}}{\sqrt{T}} \right] \leq \sqrt{c \iota} \sqrt{\frac{v}{T}} + \sqrt{c \iota} \frac{v}{\sqrt{T}}.
\end{align*}
Dividing both sides by $\sqrt{c \iota}$,
\begin{align*}
\frac{v + \sqrt{v}}{\sqrt{T}} \geq \frac{u}{\sqrt{c \iota}} - \left[ \frac{\sqrt{\abs{u-v}}}{\sqrt{T}} + T^{\frac{1}{4}} \sqrt{u \abs{u-v}} + \frac{\abs{u-v}}{\sqrt{T}}   \right]
\end{align*}
Put $v = T^2$ and $u = T^2+1$, then we get
\begin{align*}
  \sqrt{T}+1 \geq \frac{T^2+1}{\sqrt{c \iota}} - \left[ \frac{1}{\sqrt{T}} + T^{\frac{1}{4}} \sqrt{T^2 + 1} + \frac{1}{\sqrt{T}}  \right].
\end{align*}
For large enough $T$, since $\iota = \mathrm{poly}(\log T)$, this inequality does
not hold. Therefore we get our contradiction.
\end{proof}

\subsection{Proof of \Cref{thm:restarted-sgd}}

\begin{theorem}\label{thm:non-conv-adapt-sgd}
(\citep{{liu23_high_probab_conver_stoch_gradien_method}}, High-probability
convergence of SGD in the nonconvex setting). Let $f$ be $L$-smooth and
possibly nonconvex. Suppose that the stochastic gradient noise is $R^2$-sub-gaussian. Then for any fixed stepsize $\eta$ such that $\eta L \leq 1$ we have
\begin{align*}
  \frac{1}{T} \sum_{t=0}^{T-1} \sqn{\nabla f(x_t)} \leq \frac{2 (f(x_0) - f_{\ast})}{\eta T} + 5 \eta R^2 + \frac{12 R^2 \log \frac{1}{\delta}}{T}.
\end{align*}
\end{theorem}
\begin{proof}
This is a very straightforward generalization of \citep[Theorem
4.1]{liu23_high_probab_conver_stoch_gradien_method}, and we include it for completeness. By \citep[Corollary 4.4]{liu23_high_probab_conver_stoch_gradien_method} we have
that if $\eta_t L \leq 1$ and $0 \leq w_t \eta_t^2 L \leq \frac{1}{2 R^2}$
\begin{align}
\label{eq:24}
  \sum_{t=1}^T \left[ w_t \eta_t \left( 1 - \frac{\eta_t L}{2} \right) - v_t \right] \sqn{\nabla f(x_t)} + w_T \Delta_{T+1} \leq w_1 \Delta_1 + \left( \sum_{t=2}^T (w_t - w_{t-1}) \Delta_t + 3 R^2 \sum_{t=1}^T \frac{w_t \eta_t^2 L}{2}  \right) + \log \frac{1}{\delta}.
\end{align}
Choose $\eta_t = \eta$ and $w_t \eta^2 L = \frac{1}{4 R^2}$, $w_t = \frac{1}{6 R^2 \eta}$.
\begin{align*}
v_t &= 3 R^2 w_t^2 \eta_t^2 (\eta_t L -1)^2 = \frac{3 R^2 \eta^2 (\eta L -1)^2}{36 R^4 \eta^2} = \frac{(1-\eta L)^2}{12 R^2}.
\end{align*}
Then
\begin{align*}
w_t \eta_t \left( 1 - \frac{\eta_t L}{2}  \right) - v_t &= \frac{1}{6 R^2} \left( 1 - \frac{\eta L}{2}  \right) - \frac{(1-\eta L)^2}{12 R^2} \\
&= \frac{1}{6 R^2} \left[ (1 - \frac{\eta L}{2}) - \frac{(1-\eta L)^2}{2}  \right] \\
&= \frac{1}{6 R^2} \left[ (1 - \frac{\eta L}{2}) - \frac{1 + \eta^2 L^2 - 2 \eta L}{2}  \right] \\
&= \frac{1}{12 R^2} \left[ 1 + \eta L - \eta^2 L^2   \right]
\end{align*}
The expression $1 + x - x^2$ is minimized for $x \in [0, 1]$ at $x=1$ and has value
$1$. Therefore
\begin{align*}
w_t \eta_t \left( 1 - \frac{\eta_t L}{2}  \right) - v_t &\geq \frac{1}{12 R^2}.
\end{align*}
Plugging into \cref{eq:24} we get
\begin{align*}
  \sum_{t=1}^T \frac{1}{12 R^2} \sqn{\nabla f(x_t)} \leq \frac{\Delta_1}{6 R^2 \eta} + \left( \frac{3 \eta}{8} T   \right) + \log \frac{1}{\delta}.
\end{align*}
Therefore
\begin{align*}
  \frac{1}{T} \sum_{t=1}^T \sqn{\nabla f(x_t)} \leq \frac{2 \Delta_1}{\eta T} + 5 \eta R^2 + \frac{12 R^2 \log \frac{1}{\delta}}{T}.
\end{align*}
\end{proof}

\subsection{Restarting SGD}

We will use the following lemma from
\citep{madden20_high_probab_conver_bound_non}:

\begin{lemma}\label{lem:subsampling-minimum}
  \citep[Lemma 33]{madden20_high_probab_conver_bound_non} Let $Z = k \in \{1, 2, \ldots, K \}$ with probability $p_k$ and $\sum_{k=1}^K p_k = 1$.
  Let $Z_1, \ldots, Z_m$ be independent copies of $Z$. Let $Y = (Y_1, \ldots, Y_m)$. Let
  $X = (X_1, \ldots, X_K)$ be a random vector on the reals independent of $Z$. Then for any
  $\gamma > 0$ we have
\begin{align*}
  \pr[ \min_{k \in Y} X_k > e \eta  ] \leq \exp (-m) + \pr[ \sum_{k=1}^K p_t X_k > \gamma ]
\end{align*}
\end{lemma}

\begin{algorithm}
\caption{FindLeader($S$, $\delta$, $K$)}
\label{alg:find-leader}
\begin{algorithmic}[1]
    \STATE \textbf{Require:} set of points $V$, desired accuracy $\delta$, and per-point estimation budget $K$.
    \STATE Set $M = \log \frac{1}{\delta}$ and let $P = \abs{V}$.
\STATE Construct the set $S = (s_1, \ldots, s_M)$ by sampling $M$ points from $v_1, \ldots, v_P$ with
replacement such that
\begin{align*}
\mathrm{Prob}(v_i \in S) &\propto \frac{1}{\sqrt{i+1}}, &&\sum_{i=1}^T \mathrm{Prob}(v_i \in S) = 1.
\end{align*}
    \FOR{$m = 1$ to $M$}
    \STATE Sample $K$ stochastic gradients $g_1^m, \ldots, g_K^m$ evaluated at $s_m$ and compute their
    average
\begin{align*}
  \hat{g}_m = \frac{1}{K} \sum_{k=1}^K g_k.
\end{align*}
\STATE Compute and store $h_m = \norm{\hat{g}_m}$.
\ENDFOR
    \STATE Find the point $s_{\mathrm{lead}} \in S$ with the minimal average stochastic
    gradient norm:
\begin{align*}
m^{\ast} = \arg\min_{m \in {1, 2, \ldots, M}} h_m, && s_{\mathrm{lead}} = S_{m^{\ast}}.
\end{align*}
    \STATE \textbf{Return} $s_{\mathrm{lead}}$ and its estimated gradient norm $g_{m^{\ast}}$.
\end{algorithmic}
\end{algorithm}

\begin{theorem}\label{thm:find-leader-conv}
(Convergence of FindLeader) If we run Algorithm~\ref{alg:find-leader} on a set
$V$ of $P$ points ${v_1, v_2, \ldots, v_P}$, with sampling budget $M$ and per-point estimation budget $K$,
then the output of the algorithm satisfies for some absolute constant $c > 0$
and all $\gamma > 0$
\begin{align*}
\pr[ \sqn{\nabla f(s_{\mathrm{lead}})} > e \gamma + c \cdot \frac{R^2 \log \frac{2dM}{\delta}}{K} ] \leq \delta + \exp(-M) + \pr[ \frac{1}{P} \sum_{p=1}^P \sqn{\nabla f(v_p)} > \gamma ].
\end{align*}
And
\begin{align}
\label{eq:34}
\norm{g_{m^{\ast}} - \nabla f(s_{\mathrm{lead}})} \leq c \cdot \frac{R^2 \log \frac{2d}{\delta}}{K}.
\end{align}
\end{theorem}
\begin{proof}
The proof of this theorem loosely follows the proofs of \citep[Theorem
2.4]{ghadimi2013stochastic} and \citep[Theorem
13]{madden20_high_probab_conver_bound_non}. First, define the following two sets of true gradients for the iterates in $V$
and $P$ respectively:
\begin{align*}
U_V = \left\{ \nabla f(v_1), \nabla f(v_2), \ldots, \nabla f(v_P) \right\} && U_S = \left\{ \nabla f(s_1), \nabla f(s_2), \ldots, \nabla f(s_M) \right\}.
\end{align*}
\Cref{lem:subsampling-minimum} gives us
\begin{align*}
  \pr[ \min_{m \in {1, 2, \ldots, M}} \sqn{\nabla f(s_m)} > e \gamma] &\leq \exp(-M) + \pr[ \frac{1}{P} \sum_{p=1}^P \sqn{\nabla f(v_p)} > \gamma ]
\end{align*}
We now compute how using the minimum from the stochastic estimates $\hat{g}_m$
affects the error. Fix $m$. Observe that because the norm of the stochastic
gradient noise $\norm{g(x) - \nabla f(x)}$ is sub-gaussian with modulus $R^2$, then
using \citep[Corollary 7]{jin19_short_note_concen_inequal_random} we get with
probability at least $1-\frac{\delta}{M}$ that for some absolute constant $c_1$
\begin{align*}
  \norm{\hat{g}_m - \nabla f(s_m)} \leq c_1 \cdot \sqrt{\frac{R^2 \log \frac{2dM}{\delta}}{K}}
\end{align*}
Squaring both sides gives
\begin{align*}
  \sqn{\hat{g}_m - \nabla f(s_m)} \leq c_1 \cdot\frac{R^2 \log \frac{2dM}{\delta}}{K}.
\end{align*}
Taking a union bound gives us that for all $m \in [M]$ we have with probability
$\delta$ that
\begin{align}
\label{eq:33}
\max_{m \in [M]} \sqn{\hat{g}_m - \nabla f(s_m)} \leq c_1 \cdot \frac{R^2 \log \frac{2dM}{\delta}}{K}.
\end{align}
We have by straightforward algebra
\begin{align*}
\min_{m \in S} \sqn{\hat{g}_m} &\leq \min_{m \in [M]} \left[ \sqn{\hat{g}_m - \nabla f(s_m) + \nabla f(s_m)} \right] \\
&\leq \min_{m \in [M]} \left[ 2 \sqn{\hat{g}_m - \nabla f(s_m)} + 2 \sqn{\nabla f(s_m)} \right] \\
&\leq \min_{m \in [M]} \left[ 2 \max_{\alpha \in [M]} \sqn{\hat{g}_{\alpha} - \nabla f(s_{\alpha})} + 2 \sqn{\nabla f(s_m)} \right] \\
&= 2 \max_{m \in [M]} \sqn{\hat{g}_m - \nabla f(s_m)} + 2 \min_{m \in [M]} \sqn{\nabla f(s_m)}.
\end{align*}
Let ${m^{\ast}}$ be the argmin. Then
\begin{align*}
\sqn{\nabla f(s_{m^{\ast}})} &\leq 2 \sqn{\nabla f(s_{m^{\ast}}) - \hat{g}_{s_{m^{\ast}}}} + 2 \sqn{\hat{g}_{s_{m^{\ast}}}} \\
&\leq 2 \sqn{\nabla f(s_{m^{\ast}}) - \hat{g}_{s_{m^{\ast}}}} + 4 \max_{m \in [M]} \sqn{\hat{g}_m - \nabla f(s_m)} + 4 \min_{m \in [M]} \sqn{\nabla f(s_m)} \\
&\leq 6 \max_{m \in [M]} \sqn{\hat{g}_m - \nabla f(s_m)} + 4 \min_{m \in [M]} \sqn{\nabla f(s_m)} \\
&\leq 6 c_1 \frac{R^2 \log \frac{2dM}{\delta}}{K} + 4 \min_{m \in [M]} \sqn{\nabla f(s_m)}.
\end{align*}
Therefore there exists some absolute constant $c$ such that
\begin{align*}
  \pr[ \sqn{\nabla f(s_{m^{\ast}})} > e \gamma + c \cdot \frac{R^2 \log \frac{2dM}{\delta}}{K} ] \leq \delta + \exp(-M) + \pr[ \frac{1}{P} \sum_{p=1}^P \sqn{\nabla f(v_p)} > \gamma ].
\end{align*}
It remains to put $s_{\mathrm{lead}} = s_{m^{\ast}}$.
\end{proof}

\begin{proof}[Proof of \Cref{thm:restarted-sgd}]
First, observe that \Cref{thm:non-conv-adapt-sgd} gives that SGD run for $T$
steps with a fixed stepsize $\eta$ such that $\eta L \leq 1$
\begin{align}
\label{eq:31}
  \frac{1}{T} \sum_{t=0}^{T-1} \sqn{\nabla f(x_t)} \leq \frac{2 (f(x_0) - f_{\ast})}{\eta T} + 5 \eta R^2 + \frac{12 R^2 \log \frac{1}{\delta}}{T}.
\end{align}
Minimizing the above in $\eta$ gives
\begin{align*}
\eta_{\ast} &= \min \left( \frac{1}{L}, \sqrt{\frac{2 (f(x_0) - f_{\ast})}{5 T R^2}} \right).
\end{align*}
We set
\begin{align*}
\eta_0 &= \min \left( \frac{1}{\overline{L}}, \sqrt{\frac{2 \underline{\Delta}}{5 T \overline{R}}} \right).
\end{align*}
Observe that $\eta_0 \leq \eta_{\ast}$. Now let
\begin{align*}
N^{\ast} &=  \ceil{\log \frac{\eta_{\ast}}{\eta_0}} \\
&= \ceil*{\log \left(  \frac{\max( \overline{L}, \sqrt{\frac{5 T \overline{R}^2}{2 \underline{\Delta}}})}{\max (L, \sqrt{\frac{5 T R^2}{\Delta}})}  \right)}.
\end{align*}
First, if we exit \Cref{alg:restarted-sgd} at line 4, i.e. if
$T_{\mathrm{total}} < N$, then by the $L$-smoothness of $f$ we have
\begin{align*}
\sqn{\nabla f(y_0)} &\leq 2L (f(y_0) - f_{\ast}) \\
&\leq N \cdot \frac{2L (f(x_0) - f_{\ast})}{T_{\mathrm{total}}} \\
&\leq \log \left(  \frac{\max( \overline{L}, \sqrt{\frac{5 T \overline{R}^2}{2 \underline{\Delta}}})}{\max (L, \sqrt{\frac{5 T R^2}{\Delta}})}  \right) \cdot \frac{L (f(x_0) - f_{\ast})}{T_{\mathrm{total}}}.
\end{align*}
This fulfills the theorem's statement. From here on our, we assume that
$N \geq T_{\mathrm{total}}$. Observe that our choice of $N$ guarantees that
$N \geq N^{\ast}$. Let $\tau$ be the first $n$ (in the loop on line 2 of
Algorithm~\ref{alg:restarted-sgd}) such that
\begin{align*}
  \frac{\eta_{\ast}}{2} \leq \eta_{\tau} \leq \eta_{\ast}.
\end{align*}
Plugging $\eta = \eta_{\tau}$ into \Cref{eq:31} we get with probability at least $\delta$ that
\begin{align}
\frac{1}{T} \sum_{t=0}^{T-1} \sqn{\nabla f(x_t^{\tau})} &\leq \frac{2 (f(x_0) - f_{\ast})}{\eta_{\tau} T} + 5 \eta_{\tau} R^2 + \frac{12 R^2 \log \frac{1}{\delta}}{T} \nonumber\\
&\leq \frac{4 (f(x_0) - f_{\ast})}{\eta_{\ast} T} + 5 \eta_{\ast} R^2 + \frac{12 R^2 \log \frac{1}{\delta}}{T} \nonumber\\
&\leq 2 \left[ \frac{2 (f(x_0) - f_{\ast})}{\eta_{\ast} T} + 5 \eta_{\ast} R^2 \right] + \frac{12 R^2 \log \frac{1}{\delta}}{T}  \nonumber\\
\label{eq:32}
&\leq 13 \left[ \sqrt{\frac{L (f(x_0) - f_{\ast}) R^2}{T}} + \frac{(f(x_0) - f_{\ast}) L}{T}  \right]  + \frac{12 R^2 \log \frac{1}{\delta}}{T}.
\end{align}
We now apply Theorem~\ref{thm:find-leader-conv} with the parameters:
\begin{align*}
V &= \left\{ x_0^{\tau}, x_1^{\tau}, \ldots, x_{T-1}^{\tau} \right\}, \\
M &= \log \frac{1}{\delta}, \\
K &= T, \\
\gamma &= 13 \left[ \sqrt{\frac{L (f(x_0) - f_{\ast}) R^2}{T}} + \frac{(f(x_0) - f_{\ast}) L}{T}  \right]  + \frac{12 R^2 \log \frac{1}{\delta}}{T}.
\end{align*}
The theorem combined with \cref{eq:32} gives us that with probability at least $1-4\delta$
\begin{align}
\label{eq:36}
\sqn{\nabla f(y_{\tau})} \leq 13 \cdot e \cdot \left[ \sqrt{\frac{L (f(x_0) - f_{\ast}) R^2}{T}} + \frac{(f(x_0) - f_{\ast}) L}{T}  \right]  + \frac{12 R^2 \log \frac{1}{\delta}}{T} + c \cdot \frac{R^2 \log \frac{2d M}{\delta}}{T}.
\end{align}
By straightforward algebra
\begin{align*}
\sqn{\hat{g}_r} = \min_{n \in [N]} \sqn{\hat{g}_n} &\leq \min_{n \in [N]} \left[ \sqn{\hat{g}_n - \nabla f(y_n) + \nabla f(y_n)} \right] \\
&\leq \min_{n \in [N]} \left[ 2 \sqn{\hat{g}_n - \nabla f(y_n)} + 2 \sqn{\nabla f(y_n)} \right] \\
&\leq \min_{n \in [N]} \left[ 2 \max_{\alpha \in [N]} \sqn{\hat{g}_{\alpha} - \nabla f(s_{\alpha})} + 2 \sqn{\nabla f(y_n)} \right] \\
&= 2 \max_{n \in [N]} \sqn{\hat{g}_n - \nabla f(y_n)} + 2 \min_{n \in [N]} \sqn{\nabla f(y_n)}.
\end{align*}
Recall that we have $r = \arg\min_{n \in [N]} \sqn{\hat{g}_n}$, then as in the proof of \Cref{thm:find-leader-conv}
we have
\begin{align}
\sqn{\nabla f(y_r)} &\leq 2 \sqn{\nabla f(y_r) - \hat{g}_{y_r}} + 2 \sqn{\hat{g}_{y_r}} \nonumber\\
&\leq 2 \sqn{\nabla f(y_r) - \hat{g}_{y_r}} + 4 \max_{n \in [N]} \sqn{\hat{g}_n - \nabla f(y_n)} + 4 \min_{n \in [N]} \sqn{\nabla f(y_n)} \nonumber\\
\label{eq:35}
&\leq 6 \max_{n \in [N]} \sqn{\hat{g}_n - \nabla f(y_n)} + 4 \min_{n \in [N]} \sqn{\nabla f(y_n)}.
\end{align}
Observe that because that we passed the budget $K=T$ to the FindLeader
procedure, we can use \Cref{eq:34} and the union bound to that with probability $1-\delta$,
\begin{align}
\label{eq:37}
\max_{n \in [N]} \sqn{\hat{g}_n - \nabla f(y_n)} \leq c \cdot \frac{R^2 \log \frac{2dN}{\delta}}{T}.
\end{align}
And clearly
\begin{align}
\label{eq:38}
\min_{n \in [N]} \sqn{\nabla f(y_n)} \leq \sqn{\nabla f(y_{\tau})}.
\end{align}
Using the estimates of \cref{eq:36,eq:37,eq:38} to upper bound the right hand
side of \cref{eq:35} gives us that with probability at least $1-5\delta$
\begin{align*}
\sqn{\nabla f(y_r)} \leq 6 c \cdot \frac{R^2 \log \frac{2dN}{\delta}}{T} + 4 \left[ 13 \cdot e \cdot \left[ \sqrt{\frac{L (f(x_0) - f_{\ast}) R^2}{T}} + \frac{(f(x_0) - f_{\ast}) L}{T}  \right]  + \frac{12 R^2 \log \frac{1}{\delta}}{T} + c \cdot \frac{R^2 \log \frac{2d M}{\delta}}{T} \right].
\end{align*}
Combining the terms and substituting in the definition of $T_{\mathrm{total}}$ gives the theorem's statement.
\end{proof}

\end{document}